\documentclass{amsart}

\usepackage{amsmath, amsthm, amssymb}
\usepackage[dvips]{graphics}

\usepackage{graphicx}

\usepackage{enumitem}
\newtheorem{theorem}{Theorem}[section]
\newtheorem{lemma}[theorem]{Lemma}
\newtheorem{lemmaa}{Lemma}
\newtheorem{lemmab}{Lemma}
\newtheorem{corollary}[theorem]{Corollary}

\theoremstyle{definition}

\theoremstyle{remark}
\newtheorem{remark}[theorem]{Remark}

\numberwithin{equation}{section}

    \usepackage{color}  
     \definecolor{red}{rgb}{0.9,0,0}
     \definecolor{green}{rgb}{0,0.6,0}
     \definecolor{rb}{rgb}{0.6,0,0.2}     
     \definecolor{pass}{rgb}{0,0,0.7}
     \definecolor{blue}{rgb}{0,0,0}

\newcommand{\pt}{\partial}

\newcommand{\R}{\mathbb{R}}

\newcommand{\LL}{{\mathcal L}}
\renewcommand{\AA}{{\mathcal A}}

\newcommand{\RR}{{\mathcal R}}

\renewcommand{\sim}{\simeq}

\newcommand {\beq} {\begin{equation}}
\newcommand {\eeq} {\end{equation}}
\newcommand {\beqa} {\begin{eqnarray}}
\newcommand {\eeqa} {\end{eqnarray}}
\newcommand {\beqann} {\begin{eqnarray*}}
\newcommand {\eeqann} {\end{eqnarray*}}

\begin{document}

\title[L1 method for a fractional-derivative problem]%
{
Error analysis of the L1 method on graded and uniform meshes for a fractional-derivative problem
in two and three dimensions}

\author{Natalia Kopteva}
\address{Department of Mathematics and Statistics,
University of Limerick, Limerick, Ireland}
\email{natalia.kopteva@ul.ie}

\thanks{The author is grateful to Dr Xiangyun Meng of Beijing Computational Science Research Center  for his helpful comments on an earlier
version of this manuscript.
The author
acknowledges financial support
from Science Foundation Ireland Grant SFI/12/IA/1683.
}
%

%

\date{}


\begin{abstract}
An initial-boundary value problem with a Caputo time derivative of fractional order $\alpha\in(0,1)$ is considered,
 solutions of which typically exhibit a singular behaviour at an initial time.
For this problem, we give
a simple framework for the analysis of the error of
L1-type discretizations on graded and uniform temporal meshes
in the $L_\infty$ and $L_2$ norms.
This framework is employed in the analysis of both finite difference and finite element spatial discretiztions.
Our theoretical findings are illustrated by numerical experiments. 
\end{abstract}

\maketitle

\section{Introduction}
The purpose of this paper is to give a simple framework for the analysis of the error in the $L_\infty(\Omega)$ and $L_2(\Omega)$ norms for   L1-type discretizations of the fractional-order parabolic problem
\beq\label{problem}
\begin{array}{l}
D_t^{\alpha}u+\LL u=f(x,t)\quad\mbox{for}\;\;(x,t)\in\Omega\times(0,T],\\[0.2cm]
u(x,t)=0\quad\mbox{for}\;\;(x,t)\in\pt\Omega\times(0,T],\qquad
u(x,0)=u_0(x)\quad\mbox{for}\;\;x\in\Omega.
\end{array}
\eeq
This problem is posed in a bounded Lipschitz domain  $\Omega\subset\R^d$ (where $d\in\{1,2,3\}$). The operator $D_t^\alpha$, for some $\alpha\in(0,1)$,
is the Caputo fractional derivative in time
defined \cite{Diet10} by
\begin{equation}\label{CaputoEquiv}
D_t^{\alpha} u(\cdot,t) :=  \frac1{\Gamma(1-\alpha)} \int_{0}^t(t-s)^{-\alpha}\, \pt_s u(\cdot, s)\, ds
    \qquad\text{for }\ 0<t \le T,
\end{equation}
where $\Gamma(\cdot)$ is the Gamma function, and $\pt_s$ denotes the partial derivative in $s$.
The spatial operator $\LL$ is a linear second-order elliptic operator:
\beq\label{LL_def}
\LL u := \sum_{k=1}^d \Bigl\{-\pt_{x_k}\!(a_k(x)\,\pt_{x_k}\!u) + b_k(x)\, \pt_{x_k}\!u \Bigr\}+c(x)\, u,
\eeq
with sufficiently smooth coefficients $\{a_k\}$, $\{b_k\}$ and $c$ in $C(\bar\Omega)$, for which we assume that $a_k>0$ in $\bar\Omega$,
and also either $c\ge 0$ or $c-\frac12\sum_{k=1}^d\pt_{x_k}\!b_k\ge 0$.
{\color{blue}All our results also apply to the case $\LL=\LL(t)$, while some remain valid for a more general uniformly-elliptic $\LL$ (i.e. with mixed second-order derivatives);
 see Remark~\ref{LL_gen}.}

Throughout the paper, it will be assumed that there exists a unique solution of this problem in $C(\bar\Omega\times[0,T])$ such that $|\pt_t^l u(\cdot,t)|\lesssim 1+t^{\alpha-l}$ for $l=0,1,2$
(the notation $\lesssim$ is rigourously defined in the final paragraph of this section). This is a realistic assumption, satisfied by typical solutions of problem \eqref{problem},
in contrast to a stronger assumption $|\pt^l u(\cdot,t)|\lesssim 1$ frequently made in the literature
(see, e.g., references in \cite[Table~1.1]{laz_2fully_16}).
Indeed,
 \cite[Theorem~2.1]{stynes_too_much_reg} shows
that if a solution $u$ of \eqref{problem} is less singular than we assume (in the sense that $|\pt_t^l u(\cdot,t)|\lesssim 1+t^{\gamma-l}$ for $l=0,1,2$ with any $\gamma>\alpha$),
then the initial condition $u_0$ is uniquely defined by the other data of the problem, which is clearly too restrictive.
At the same time, our results can be easily applied to the case of $u$ having no
singularities or exhibiting a somewhat different singular behaviour at $t=0$.

We consider L1-type schemes for problem \eqref{problem}, which employ the discetization of  $D^\alpha_tu$ defined, for $m=1,\ldots,M$, by
\vspace{-0.1cm}
\beq\label{delta_def}
\delta_t^{\alpha} U^m :=  \frac1{\Gamma(1-\alpha)} \sum_{j=1}^m \delta_t U^j\!\int_{t_{j-1}}^{t_j}\!\!(t_m-s)^{-\alpha}\, ds,
\qquad
\delta_t U^j:=\frac{U^j-U^{j-1}}{t_j-t_{j-1}},\vspace{-0.1cm}%
\eeq
when associated with the temporal mesh $0=t_0<t_1<\ldots <t_M=T$ on $[0,T]$.
Similarly to \cite{stynes_etal_sinum17}, our main interest will be in graded temporal meshes as they offer an efficient way of computing reliable numerical approximations of solutions singular at $t=0$.
We shall also consider uniform temporal meshes, as although the latter have lower convergence rates near $t=0$, they have been shown to be first-order accurate for $t\gtrsim 1$
\cite{gracia_etal_cmame,laz_L1}.
\smallskip

{\it Novelty.} {\color{blue}
We present a new framework for the estimation of the 
error whenever an L1 scheme is used on graded or uniform temporal meshes.
This framework is simple, applies to both finite difference and finite element spatial discretizations, and works for error estimation in both  $L_2(\Omega)$ and  $L_\infty(\Omega)$ norms.
It easily extends to general elliptic operators $\LL=\LL(t)$, as well as  quasi-uniform and quasi-graded temporal meshes.
Naturally, it yields versions of some previously-known error bounds as particular cases. It is also used here to establish entirely new results.


Graded meshes for problem of type \eqref{problem} for the case $d=1$
were recently considered  in
\cite{stynes_etal_sinum17}, where maximum  norm error bounds are obtained for finite difference discretizations.
In comparison, our analysis deals with temporal-discretization errors on graded meshes in an entirely different and substantially more concise way.
To be more precise, we use more intuitive integral representations of the temporal truncation errors; see Lemma~\ref{lem_simplest}.
Once error bounds on graded meshes are established for a paradigm problem without spatial derivatives,
they seamlessly extend to finite difference and finite element spatial discretizations  of  \eqref{problem} for any $d\ge 1$.
Our results on graded meshes are new for finite element discretizations, as well as for finite difference discretizations for $d>1$.

The convergence behaviour of the L1 method on uniform temporal meshes is well-understood. In particular,
for finite element spatial discretizations,
the errors
in the $L_2(\Omega)$ norm have been estimated in  \cite{laz_L1} using Laplace transform techniques (for $\LL=-\triangle$ and $f=0$).
For finite difference discretizations for $d=1$,
a similar error bound the maximum norm  was established in \cite{gracia_etal_cmame}.
Within our theoretical framework, we easily get versions of error bounds of \cite{laz_L1} and \cite{gracia_etal_cmame}.
Furthermore, we give error bounds for finite element discretizations in the
$L_\infty(\Omega)$ norm on uniform temporal meshes, which appear to be entirely  new.
(Some error bounds in the $L_\infty$ norm for linear-finite-element spatial semi-discretizations are given in \cite{Must18}.)

Our approach to uniform meshes is very similar to the case of graded meshes. The main difference is in that now we employ a more subtle
stability property of the discrete fractional-derivative operator $\delta_t^\alpha$   from
 \cite{gracia_etal_cmame}, a version of which is also given in \cite{JZ17}; see Lemma~\ref{lem_main_stability_star}.
Additionally, we give a considerably shorter and more intuitive proof of this stability result.
This new proof relies on 
a simple barrier function, and may be of independent interest; see Appendix~\ref{app_B}.}
\smallskip


{\it Outline.}
We start by  presenting, in \S\ref{sec_paradigm}, a paradigm for the temporal-error analysis
using a simplest example without spatial derivatives.
This error analysis is extended in \S\ref{sec_semidiscr} to temporal semidiscretizations of \eqref{problem}.
Full discretizations that employ finite differences and finite elements
are respectively addressed
 in \S\ref{sec_FD} and \S\ref{sec_FE}.
Finally, the assumptions on the derivatives of the exact solution are  discussed in \S\ref{sec_dervts},
and our theoretical findings are illustrated by numerical experiments in \S\ref{sec_Num}.
\smallskip

{\it Notation.}
We write
 $a\sim b$ when $a \lesssim b$ and $a \gtrsim b$, and
$a \lesssim b$ when $a \le Cb$ with a generic constant $C$ depending on $\Omega$, $T$, $u_0$ and
$f$,
but not 
%
 on the total numbers of degrees of freedom in space or time.
  Also, for 
  $1 \le p \le \infty$, and $k \ge 0$,
  we shall use the standard norms
  in the spaces $L_p(\Omega)$ and the related Sobolev spaces $W_p^k(\Omega)$,
  while  $H^1_0(\Omega)$ is the standard space of functions in $W_2^1(\Omega)$ vanishing on $\pt\Omega$.

%


\section{Paradigm for the temporal-discretization error analysis}\label{sec_paradigm}

\subsection{Graded temporal mesh}
Throughout the paper, we shall frequently consider the graded temporal mesh
$\{t_j=T(j/M)^r\}_{j=0}^M$ with some $r\ge 1$ (while $r=1$ generates a uniform mesh).
For this mesh, a calculation shows that
\beq\label{t_grid}
\tau_j:=t_j-t_{j-1}\simeq M^{-1}\,t_j^{1-1/r}
\qquad
\mbox{for\;\;}j=1,\ldots,M.
\eeq
This follows from $\tau_1=t_1\simeq M^{-r}$ for $j=1$,
and $t_j\le 2^{r}  t_{j-1}$
for $j\ge 2$.

{\color{blue}Note that all  results of the paper immediately apply to a quasi-graded mesh defined by $\{t_j=T(\xi_j)^r\}_{j=0}^M$, where $\{\xi_j\}_{j=0}^M$
is a quasi-uniform mesh on $[0,1]$.}

\subsection{Stability properties of the discrete fractional operator $\delta_t^\alpha$}
The definition \eqref{delta_def} of $\delta_t^\alpha$ can be rewritten as
\begin{subequations}\label{kappa_def}
\begin{align}
\label{kappa_def_a}
\delta_t^\alpha V^m &= \underbrace{{\kappa_{m,m}}}_{{}>0} V^m-\sum_{j=1}^m \underbrace{(\kappa_{m,j}-\kappa_{m,j-1})}_{{}>0}V^{j-1},
\\
\label{kappa_def_b}
\kappa_{m,j}&:= \frac{\tau_j^{-1}}{\Gamma(1-\alpha)} \int_{t_{j-1}}^{t_{j}}\!\! (t_m-s)^{-\alpha}\, ds
\quad \mbox{for}\;\; j=1,\ldots,m,\quad\;\; \kappa_{m,0}:=0.
\end{align}
Here
$\kappa_{m,j}$ for $j\ge 1$ is the average of the function $\{\Gamma(1-\alpha)\}^{-1}(t_m-s)^{-\alpha}$ on the interval $s\in(t_{j-1},t_j)$, so
 $\kappa_{m,j-1}\le \kappa_{m,j}$ for all admissible  $j$ and $m$.
\end{subequations}
\smallskip

\begin{lemma}\label{lem_main_stability}
(i)
For any $\{V^j\}_{j=0}^M$ on an arbitrary mesh $\{t_j\}_{j=0}^M$,  one has
$$
|V^m-V^0|\lesssim \max_{j=1,\ldots,m}\bigl\{t_j^{\alpha}\, |\delta_t^\alpha V^j|\bigr\}
\qquad\mbox{for}\;\; m=1,\ldots M.
$$
{\color{blue}(ii)
If $V^0=0$ and $\delta_t^\alpha |V^j|\le |F^j|$ for $j=1,\ldots,M$, then
$|V^m|\lesssim \max_{j=1,\ldots,m}\bigl\{t_j^{\alpha}\, |F^j|\bigr\}$
for $m=1,\ldots,M$.}
\end{lemma}

\begin{proof}
(i)
Let $W^j:=V^j-V^0$; then $W^0=0$, while $\delta_t^\alpha W^j=\delta_t^\alpha V^j=:F^j$,
so  
we need to prove that
$|W^m|\lesssim \max_{j\le m}\{t_j^{\alpha}\,|F^j|\}$.
Let $\max_{j\le m}|W^j|=|W^n|$ for some $1\le n\le m$.
Then, by \eqref{kappa_def_a} combined with $W^0=0$, one gets
\beq\label{Wn_estimation}
 \underbrace{{\kappa_{n,n}}}_{{}>0} |W^n|- \sum_{j=2}^n \underbrace{(\kappa_{n,j}-\kappa_{n,j-1})}_{{}>0}|W^n|\le|F^n|
 \quad\Rightarrow\quad
|W^n|\le \kappa_{n,1}^{-1}\,|F^n|.
\eeq
Next, recalling \eqref{kappa_def_b}, and also using $(t_n-s)^{-\alpha}\ge t_n^{-\alpha}$ on $(0,t_1)$, 
one concludes that $\kappa_{n,1}\gtrsim t_n^{-\alpha}$.
So $|W^n|\lesssim t_n^\alpha\,|F^n|$, which immediately implies the desired assertion.
%

{\color{blue}(ii) Let $W^0=0$ and $\delta_t^\alpha W^j= |F^j|$ for $j=1,\ldots,M$.
Then $0\le |V^m|\le W^m$
(as $\delta_t^\alpha$ is associated with an $M$-matrix), while $|W^m|\lesssim \max_{j=1,\ldots,m}\bigl\{t_j^{\alpha}\, |F^j|\bigr\}$
by the result of part (i). The desired assertion follows.
}
\end{proof}

To deal with uniform temporal meshes, we employ a more subtle stability result.

\renewcommand{\thelemmab}{{\ref{lem_main_stability}}${}^*$}
\begin{lemmab}[{\cite{gracia_etal_cmame}}]\label{lem_main_stability_star}
Let $r=1$ and $\tau:=TM^{-1}$.
Given $\gamma\in(0,\alpha]
$, if $V^0=0$ and $|\delta_t^\alpha V^j|\lesssim \tau^\gamma t_j^{-\gamma-1}$ for $j=1,\ldots,M$, then
$|V^j|\lesssim t_j^{\alpha-1}$ for $j=1,\ldots, M$.
\end{lemmab}

\begin{proof}
The desired assertion follows from \cite[Lemma~3]{gracia_etal_cmame} with $\beta=1+\gamma$; {\color{blue}see also \cite[Theorem~3.3]{JZ17} for a similar result}.
We give an alternative (substantially shorter) proof in Appendix~\ref{app_B}.
\end{proof}

The next lemma will be useful when dealing with Ritz projections while estimating the errors of finite element discretizations in \S\ref{sec_FE}.

\begin{lemma}\label{lem second_stability}
Let $\{V^j\}_{j=0}^M\in \R^{M+1}$ and $\{\lambda^j\}_{j=1}^M\in \R^{M}$, and
$\bar\lambda=\bar\lambda(t)$ be a piecewise-constant left-continuous function defined by
$\bar\lambda(t)=\lambda^j$ for $t\in(t_{j-1},t_j]$, $j=1,\ldots,M$.
Then, with the  notation
$J^{1-\alpha}v(t):=\{\Gamma(1-\alpha)\}^{-1}\!\! \int_{0}^t(t-s)^{-\alpha} v( s)\, ds$,
\beq\label{rho_imply}
\delta_t^\alpha V^j \le J^{1-\alpha}\bar\lambda(t_j)\quad\forall\,j\ge1
\quad\;\;\Rightarrow\quad\;\;
V^m-V^0\le 
\sum_{j=1}^m\tau_j\, \lambda^j\quad\forall\,m\ge 0.
\eeq
\end{lemma}

\begin{proof}
Let $\Lambda^j:=V^0+\int_0^{t_{j}}\bar\lambda\,dt$ so that $\lambda^j=\delta_t\Lambda^j$.
Now, $J^{1-\alpha}\bar\lambda(t_j)=\delta_t^\alpha\Lambda^j$,
so we get $M$  
equations 
$\delta_t^\alpha V^j \le \delta^\alpha_t\Lambda^j$ for $j=1,\ldots,M$.
Augmenting these equations by $V^0=\Lambda^0$,
we get the matrix relation $A\vec{V}\le A\vec{\Lambda}$
for the column vectors $\vec{V}:=\{V^j\}_{j=0}^M$ and $\vec{\Lambda}:=\{\Lambda^j\}_{j=0}^M$
with an inverse-monotone $(M+1)\times (M+1)$ matrix $A$. (The latter follows from $A$ being diagonally dominant, with the entries $A_{ij}\le 0$ for $i\neq j$
in view of \eqref{kappa_def_a}.)
Consequently, $\vec{V}\le \vec{\Lambda}$, which immediately yields the desired assertion.
\end{proof}

\subsection{Error estimation for a simplest example\! (without spatial derivatives)}
It is convenient to illustrate our approach to the estimation of the temporal-discretization error using a very simple example.
Consider a fractional-derivative problem without spatial derivatives together with its discretization:
\begin{subequations}\label{simplest}
\begin{align}
D_t^\alpha u(t)&=f(t)&&\hspace{-1.6cm}\mbox{for}\;\;t\in(0,T],&&\hspace{-0.8cm} u(0)=u_0,
\\ \delta_t^\alpha U^j&=f(t_j)&&\hspace{-1.6cm}\mbox{for}\;\;j=1,\ldots,M,&&\hspace{-0.8cm} U^0=u_0.
\end{align}
\end{subequations}
Throughout this subsection, with slight abuse of notation, $\pt_t $ will be used for $\frac{d}{dt}$, while
$\delta_tu(t_j):=\tau_j^{-1}[u(t_j)-u(t_{j-1})]$ (similarly to $\delta_t$ in \eqref{delta_def}).

\begin{lemma}\label{lem_simplest}
Let  $\{t_j=T(j/M)^r\}_{j=0}^M$ for some $r\ge1$.
Then for $u$ and $U^j$ that satisfy \eqref{simplest}, one has
$$
|u(t_m)-U^m|
\lesssim
\max_{j=1,\ldots,m}\psi^j,
$$
where $m=1,\ldots,M$, and
\begin{subequations}\label{psi_def}
\begin{align}\label{psi_1_def}
\psi^1&:=\tau_1^\alpha\sup_{s\in(0,t_1)}\!\!\bigl(s^{1-\alpha}|\delta_tu(t_1)-\pt_s u(s)|\bigr),
\\\label{psi_j_def}
\psi^j&:=\tau_j^{2-\alpha} \,t_j^\alpha \sup_{s\in(t_{j-1},t_j)}\!\!\!|\pt_s^2u(s)|\qquad\mbox{for}\;\;j\ge 2.
\end{align}
\end{subequations}

\end{lemma}

\begin{proof}
{\color{blue}Using the standard piecewise-linear Lagrange interpolant $u^I$ of $u$, let}
$$
\chi:=u-u^I
\quad\Rightarrow\quad
|\chi(s)|\le \underbrace{\tau_j(t_j-s)}_{{}\le \tau_j^2}\,\sup_{s\in(t_{j-1},t_j)}\!\!\!|\pt_s^2u|
\quad\mbox{for~}s\in[t_{j-1},t_j].
$$
{\color{blue}As $\chi$ will appear in the truncation error, it is useful
to  note that, in view of \eqref{psi_j_def},}~
\begin{subequations}\label{chi_bounds}
\beq\label{chi_bounds_j2}
|\chi(s)|\le \psi^j\, \tau_j^\alpha \,t_j^{-\alpha}\min\{1,{\color{blue}(t_j-s)/\tau_j}\}\quad\;\mbox{for}\;\; s\in(t_{j-1},t_j),\;j\ge2.
\eeq
On $(0,t_1)$, 
one has $\chi'(s)=\pt_s u(s)-\delta_tu(t_1)$, which,
combined with \eqref{psi_1_def}, 
yields
\beq
|\chi(s)|\le \psi^1\,\tau_1^{-\alpha} 
\underbrace{
\int_s^{t_1}\!\!\!\zeta^{\alpha-1}\,d\zeta}_{{}\lesssim s^{\alpha-1}(t_1-s)}
\lesssim
\psi^1\, \tau_1^{-\alpha}\, s^{\alpha-1}\,(t_1-s)
\quad\;\mbox{for}\;\; s\in(0,t_1).
\eeq
\end{subequations}

We now proceed to estimating
the error $e^j:=u(t_j)-U^j$, for which \eqref{simplest} implies
\beq\label{err_simplest_prob}
\delta_t^\alpha e^j=\underbrace{\delta_t^\alpha u(t_j)-D_t^\alpha u(t_j)}_{{}=:r^j}\quad\mbox{for}\;\;j=1,\ldots,M,
\qquad e^0=0.
\eeq
For $r^m$, recalling the definitions \eqref{CaputoEquiv} and \eqref{delta_def} of $D^\alpha_t$ and $\delta_t^\alpha$, we arrive at
$$
\Gamma(1-\alpha)\,r^m\!=\!\sum_{j=1}^m\!\int_{t_{j-1}}^{t_j}\!\!\!\!(t_m-s)^{-\alpha}\!\underbrace{[\delta_t u(t_j)-\pt_s u(s)]}_{{}=-\chi'(s)}ds
=\alpha\sum_{j=1}^m\!\int_{t_{j-1}}^{t_j}\!\!\!\!(t_m-s)^{-\alpha-1}\chi(s)\,ds.
$$
(In particular, for the interval $(t_{m-1},t_m)$, to check the validity of the above integration by parts,
with $\epsilon\rightarrow 0^+$,
one can integrate by parts over $(t_{m-1},t_m-\epsilon)$.)

Next, combining the above representation of $r^m$ with the bounds \eqref{chi_bounds} on $\chi$,
{\color{blue} we claim that}
\beq\label{r_m_simplest}
|r^m|\lesssim\mathring{\mathcal J}^m\,(\tau_1/t_m)\,\psi^1+
{\mathcal J}^m\max_{j=2,\ldots,m}\{\nu_{m,j}\,\psi^j\},
\eeq
where
\begin{align*}
\mathring{\mathcal J}^m&:=\tau_1^{-\alpha}\,(t_m/\tau_1)\int_0^{t_1}\!\!s^{\alpha-1}
(t_1-s)\,
(t_m-s)^{-\alpha-1}
ds ,
\\
{\mathcal J}^m&:=\tau_m^\alpha\,\, t_m^{-\alpha(1-1/r)}\int_{{\color{blue}t_1}}^{t_m}\!\!s^{-\alpha/r}\,(t_m-s)^{-\alpha-1}\,\min\{1,(t_m-s)/\tau_m\}\,ds,
\\[0.2cm]
\nu_{m,j}&:=(\tau_j/\tau_m)^\alpha\,(t_j/t_m)^{-\alpha(1-1/r)} \simeq 1. 
\end{align*}
Here, the  bound on $\nu_{m,j}$ follows from $\tau_j/\tau_m\simeq (t_j/t_m)^{1-1/r}$ (in view of \eqref{t_grid}).
{\color{blue}To check the bound \eqref{r_m_simplest}, note that the two terms in its right-hand side are respectively associated with $\int_0^{t_1}$ and $\sum_{j=2}^m\int_{t_{j-1}}^{t_j}$
in $r^m$. Note also that for $j=2,\ldots, m-1$, it is convenient to use a version of \eqref{chi_bounds_j2} with
$\min\{1,(t_j-s)/\tau_j\}\le 1$ replaced by $\min\{1,(t_m-s)/\tau_m\}\ge 1$.
So a calculation using \eqref{chi_bounds_j2} and the definition of $\nu_{m,j}$ implies for $j= 2,\ldots, m$ that
$$
|\chi(s)|\le \{\nu_{m,j}\,\psi^j\}\,\tau_m^\alpha \,t_m^{-\alpha(1-1/r)}\,\underbrace{t_j^{-\alpha/r}}_{\lesssim s^{-\alpha/r}}\, \min\{1,(t_m-s)/\tau_m\}\quad\;\mbox{for}\;\; s\in(t_{j-1},t_j).
$$
This observation leads to the definition of ${\mathcal J}^m$ in \eqref{r_m_simplest}.}

For $\mathring{\mathcal J}^m$, the observation that $(t_1-s)/(t_m-s)\le t_1/t_m$ for $s\in(0,t_1)$ implies
$$
\mathring{\mathcal J}^m
\le
t_m^{-\alpha}\int_0^{t_1}\!\!s^{\alpha-1}(t_1-s)^{-\alpha}\, ds
=t_m^{-\alpha}\int_0^{1}\!\hat s^{\alpha-1}(1-\hat s)^{-\alpha}
\, d\hat s
\lesssim t_m^{-\alpha},
$$
where $\hat s:=s/t_1$.
For  ${\mathcal J}^m$, it is helpful to employ another substitution $\hat s:=s/t_m$ and $\hat\tau_j:=\tau_j/t_m$, so,
for $m\ge2$, one gets
$$
{\mathcal J}^m=t_m^{-\alpha}\, \hat\tau_m^\alpha\,\underbrace{  \int_{{\color{blue}\hat\tau_1}}^{1}\!\!\hat s^{-\alpha/r}(1-\hat s)^{-\alpha-1}\min\{1,(1-\hat s)/\hat\tau_m\}\,d\hat s}_{{}\lesssim \hat\tau_m^{-\alpha}}
\lesssim t_m^{-\alpha}.
$$
{\color{blue}Here, when bounding the integral, it is convenient to replace the lower limit $\hat\tau_1$ by~$0$,
and then consider the intervals $(0,2^{-r} )$, $(2^{-r},1-\hat\tau_m)$ and $(1-\hat\tau_m,1)$ separately (in view of $1-\hat\tau_m\ge 2^{-r}$).
On these intervals, the integrand is respectively $\lesssim \hat s^{-\alpha/r}$, $\lesssim(1-\hat s)^{-\alpha-1}$ and $\lesssim(1-\hat s)^{-\alpha}/\hat\tau_m$, so the corresponding integrals
are respectively $\lesssim 1$ (in view of $\alpha/r\in(0,1)$), $\lesssim\hat\tau_m^{-\alpha}$ and $\lesssim\hat\tau_m^{-\alpha}$.
So the above bound on ${\mathcal J}^m$ is indeed true.}

Finally, we combine \eqref{r_m_simplest} with the above bounds on $\mathring{\mathcal J}^m$ and ${\mathcal J}^m$, and arrive at
\beq\label{r_m_simplest_main}
|r^m|\lesssim t_m^{-\alpha}\bigl\{(\tau_1/t_m)\psi^1+\max_{j=2,\ldots,m}\psi^j\bigr\},
\eeq
while $|\delta_t^\alpha e^m|=|r^m|
$.
As $\tau_1/t_m\le 1$, the desired assertion
follows by an application of Lemma~\ref{lem_main_stability}.
\end{proof}

\begin{corollary}\label{cor_simplest}
Under the conditions of Lemma~\ref{lem_simplest}, suppose $|\pt_t^l u(t)|\lesssim 1+t^{\alpha-l}$ for $l=1,2$ and $t\in(0,T]$.
Then $|u(t_m)-U^m|\lesssim M^{-\min\{\alpha r,2-\alpha\}}$ for $m=1,\ldots,M$.
\end{corollary}

\begin{proof}
It suffices to show that $\psi^j\lesssim M^{-\min\{\alpha r,2-\alpha\}}$ for $j\ge 1$.
As $t\le T$, we have $|\pt_t^l u(t)|\lesssim t^{\alpha-l}$.
For $\psi^1$ of \eqref{psi_1_def}, note that $s^{1-\alpha}|\delta_tu(t_1)|\lesssim \tau_1^{-\alpha}\int_0^{\tau_1}s^{\alpha-1}\simeq 1$,
while $s^{1-\alpha}|\pt_su(s)|\lesssim 1$, so
$\psi^1\lesssim \tau_1^{\alpha}\simeq M^{-\alpha r}$.
For any other $\psi^j$, defined in \eqref{psi_j_def}, in view of $t_{j-1}\ge 2^{-r}t_j$, one gets
$|\pt^2_s u(s)|\lesssim t_j^{\alpha-2}$ for $s\in(t_{j-1},t_j)$, so
$\psi^j\lesssim (\tau_j/t_j)^{2-\alpha}t_j^{\alpha}$.
Now, set $\gamma:=\min\{\alpha r,2-\alpha\}$.
Then
$(\tau_j/t_j)^{2-\alpha}\le (\tau_j/t_j)^{\gamma}\lesssim M^{-\gamma} t_j^{-\gamma/r}$,
by~\eqref{t_grid}.
Combining this with $ t_j^{\alpha-\gamma/r}\lesssim 1$ yields $\psi^j\lesssim M^{-\gamma}=M^{-\min\{\alpha r,2-\alpha\}}$ for $j\ge2$.
\end{proof}

\begin{remark}[Optimal mesh grading $r$]\label{rem_r_optimal}
The optimal error bound $O(M^{-(2-\alpha)})$ in Corollary~\ref{cor_simplest} is attained when $r= (2-\alpha)/\alpha$.
For any larger $r$, one also enjoys the optimal
rate of convergence; however,  increased temporal mesh widths near $t=T$ (for example, $\tau_M\approx rTM^{-1}$) lead to larger errors.
See also \cite[Remark~5.6]{stynes_etal_sinum17}.
\end{remark}

\subsection{Analysis on the uniform mesh.}
Let us now consider the case of a uniform temporal mesh (i.e. $r=1$).
If $u$ is smooth on $[0,T]$ in the sense that $|\pt^l_t u|\lesssim 1$ for $l=1,2$, then an application of Lemma~\ref{lem_simplest} immediately yields for  the error to be${}\lesssim M^{-(2-\alpha)}$. However,
we are interested in a more realistic case of $u$ being singular at $t=0$.
%

We start with a shaper version of Lemma~\ref{lem_simplest}.

\renewcommand{\thelemmaa}{\ref{lem_simplest}${}^*$}
\begin{lemmaa}\label{lem_simplest_star}
Under the conditions of Lemma~\ref{lem_simplest},
let $r=1$ and $\tau:=TM^{-1}$, and set $\gamma=\min\{\alpha,1-\alpha\}$. Then
$$
|u(t_m)-U^m|
\lesssim t_m^{\alpha-1}
\max_{j=1,\ldots,m}\bigl\{\tau^{-\gamma}\,t_j^{1-\alpha+\gamma}\,\psi^j\bigr\}.
$$
\end{lemmaa}\vspace{-0.5cm}
\begin{proof}
An inspection of the proof of Lemma~\ref{lem_simplest} shows that one can replace
the term ${\mathcal J}^m\max_{j=2,\ldots,m}(\nu_{m,j}\,\psi^j)$ in \eqref{r_m_simplest} (where recall that $\nu_{m,j}\sim1$) by
\beq\label{widetilde_J}
\widetilde{\mathcal J}^m\max_{j=2,\ldots,m}\bigl\{(t_j/t_m)
\,\psi^j\bigr\}
,
\eeq
%
%
where (with the use of $t_j^{-1}\le s^{-1}$ for $s\in(t_{j-1},t_j)$)
$$
\widetilde{\mathcal J}^m:=\tau_m^\alpha\, t_m^{-\alpha(1-1/r)+\fbox{\scriptsize$1$}}\int_{t_1}^{t_m}\!
s^{-\alpha/r-\fbox{\scriptsize$1$}}
\,(t_m-s)^{-\alpha-1}\,\min\{1,(t_m-s)/\tau_m\}\,ds.
$$
%
Here, for convenience, the terms that differ from ${\mathcal J}^m$ are framed.

{\color{blue}Next, set $r=1$ and $\tau_j=\tau$.}
%
We claim that
$\widetilde{\mathcal J}^m\lesssim t_m^{-\alpha}$ for $m\ge2$.
Indeed, imitating the estimation of ${\mathcal J}^m$ in the proof of Lemma~\ref{lem_simplest},
we employ the substitution $\hat s:=s/t_m$ and the notation
{\color{blue}$\hat\tau:=\tau/t_m$ to get
$$
\widetilde{\mathcal J}^m=t_m^{-\alpha}\, \hat\tau^\alpha\,\underbrace{  \int_{\hat \tau}^{1}\!\!
\hat s^{\color{blue}-\alpha-1}(1-\hat s)^{-\alpha-1}\min\{1,(1-\hat s)/\hat\tau\}\,d\hat s}_{{}\lesssim
\hat\tau^{-\alpha} }
\lesssim t_m^{-\alpha}.
$$
Here $\hat\tau \le \frac12\le1-\hat\tau$, so one may consider the intervals $(\hat \tau,\frac12 )$, $(\frac12,1-\hat\tau)$ and $(1-\hat\tau,1)$ separately.
}

Now, using \eqref{widetilde_J} in \eqref{r_m_simplest},
we arrive at
a version of
\eqref{r_m_simplest_main}:
\beq\label{new_rm_bound}
|r^m|\lesssim t_m^{-\alpha-1}\max_{j=1,\ldots,m}\bigl\{t_j\psi^j\bigr\}
\lesssim  t_m^{-\gamma-1}\max_{j=1,\ldots,m}\bigl\{t_j^{1-\alpha+\gamma}\,\psi^j\bigr\}.
\eeq
Finally, 
an application of Lemma~\ref{lem_main_stability_star} yields the desired assertion.
\end{proof}


\begin{corollary}[Uniform temporal mesh]\label{cor_simplest_star}
Under the conditions of Lemma~\ref{lem_simplest}, let $r=1$ and $\tau=T M^{-1}$, and suppose $|\pt^l_t u(t)|\lesssim 1+t^{\alpha-l}$  for $l=1,2$ and $t\in(0,T]$.
Then $|u(t_m)-U^m|\lesssim t_m^{\alpha-1}M^{-1}\lesssim M^{-\alpha}$ for $m=1,\ldots,M$.
\end{corollary}
\begin{proof}
We imitate the proof of Corollary~\ref{cor_simplest}, only now
employ Lemma~\ref{lem_simplest_star}.
So it suffices to show that $\tau^{-\gamma}\,t_j^{1-\alpha+\gamma}\psi^j\lesssim \tau$.
For $j=1$, this follows from $\psi^1\lesssim \tau^\alpha$,
%
while for $j\ge2$, from
$\psi^j\lesssim \tau^{2-\alpha} t_j^{\alpha+(\alpha-2)}\lesssim \tau^{1+\gamma} t_j^{\alpha-1-\gamma}$
(as $\tau\le t_j$ and $\gamma\le 1-\alpha$).
\end{proof}

\section{Error analysis for the L1 semidiscretization in time}\label{sec_semidiscr}
Consider the semidiscretization of our problem~\eqref{problem} in time using the L1-method:
\beq\label{semediscr_method}
\delta_t^\alpha U^j +\LL U^j= f(\cdot,t_j)\;\;\mbox{in}\;\Omega,\;\;U^j=0\;\;\mbox{on}\;\pt\Omega\;\;\mbox{for}\;\;j=1,\ldots,M;\quad U^0=u_0.
\eeq

\begin{theorem}\label{theo_semidiscr}
(i) Given $p\in\{2,\infty\}$, let  $\{t_j=T(j/M)^r\}_{j=0}^M$ for some $r\ge1$,
and $u$ and $U^j\!$  respectively satisfy \eqref{problem},\eqref{LL_def}\! and \eqref{semediscr_method}.
Then, under the condition \mbox{$c-p^{-1}\!\sum_{k=1}^d\!\pt_{x_k}\!b_k\ge 0$}, one has
\beq\label{semediscr_error_bound}
\|u(\cdot,t_m)-U^m\|_{L_p(\Omega)}\lesssim \max_{j=1,\ldots, m}\|\psi^j\|_{L_p(\Omega)}
\qquad \mbox{for}\;\;
m=1,\ldots,M,
\eeq
where
$\psi^j=\psi^j(x)$ is defined by \eqref{psi_def}, in which $u(\cdot)$ is understood as $ u(x,\cdot)$  when evaluating $\pt_su$, $\pt_s^2 u$ and $\delta_tu$.
\\
(ii) Furthermore,
if $r=1$,
a sharper  $\max_{j=1,\ldots, m}\{\tau^{-\gamma}\,t_j^{1-\alpha+\gamma}\,\|\psi^j\|_{L_p(\Omega)}\}$
can replace the right-hand side  in \eqref{semediscr_error_bound},
where $\tau=TM^{-1}$ and $\gamma=\min\{\alpha,1-\alpha\}$.
\end{theorem}

\begin{corollary}\label{cor_semidiscr}
(i)
Under the conditions of Theorem~\ref{theo_semidiscr}, suppose $\|\pt_t^l u(\cdot,t)\|_{L_p(\Omega)}\lesssim 1+t^{\alpha-l}$ for $l=1,2$ and $t\in(0,T]$.
Then $\|u(\cdot,t_m)-U^m\|_{L_p(\Omega)}\lesssim M^{-\min\{\alpha r,2-\alpha\}}$ for $m=1,\ldots,M$.
\\
(ii) If, additionally, $r=1$, then  $\|u(\cdot,t_m)-U^m\|_{L_p(\Omega)}\lesssim t_m^{\alpha-1}M^{-1}$ for $m=1,\ldots,M$.
\end{corollary}

\begin{proof}
Imitate the proofs of Corollaries~\ref{cor_simplest} and~\ref{cor_simplest_star} for parts (i) and (ii), respectively.
\end{proof}


\noindent
{\it Proof of Theorem~\ref{theo_semidiscr}.}
For the error $e^m:= u(\cdot,t_m)-U^m$, using \eqref{problem} and \eqref{semediscr_method}, one easily gets a version of \eqref{err_simplest_prob}:
\beq\label{new_e_prob_}
\delta_t^\alpha e^m +\LL e^m=
\underbrace{\delta_t^\alpha u(\cdot,t_m)-D_t^\alpha u(\cdot,t_m)}_{{}=:r^m}\quad\mbox{for}\;\;m=1,\ldots,M,
\qquad e^0=0.
\eeq
Note that the bound \eqref{r_m_simplest_main} on $r^m$  obtained in the proof of Lemma~\ref{lem_simplest} implies that
$\|r^m\|_{L_p(\Omega)}\lesssim t_m^{-\alpha}\max_{j=1,\ldots,m}\|\psi^j\|_{L_p(\Omega)}$.
Hence, to complete the proof of part (i), it suffices to show that
\beq\label{semi_error}
\delta_t^\alpha \|e^m\|_{L_p(\Omega)}\le \|r^m\|_{L_p(\Omega)}
\qquad\mbox{for}\;\;m=1,\ldots, M.
\eeq
Then, indeed, \eqref{semediscr_error_bound} immediately follows by an application of Lemma~\ref{lem_main_stability}.

If $r=1$, combining
\eqref{new_rm_bound} (obtained in the proof of Lemma~\ref{lem_simplest_star})
with \eqref{semi_error} and then applying Lemma~\ref{lem_main_stability_star} yields the assertion of part (ii).

We now proceed to establishing \eqref{semi_error}.
Rewrite the equation $\delta_t^\alpha e^m +\LL e^m=r^m$ using \eqref{kappa_def_a} as
\beq\label{semidsicr_e_eq}
\underbrace{\kappa_{m,m}}_{{}>0} e^m+\LL e^m=\sum_{j=1}^m \underbrace{(\kappa_{m,j}-\kappa_{m,j-1})}_{{}>0}e^{j-1}+r^m,
\eeq
and address the cases $p=2$ and $p=\infty$ separately.

For $p=2$, consider the $L_2(\Omega)$ inner product (denoted $\langle\cdot,\cdot\rangle$)
of \eqref{semidsicr_e_eq} with $e^m$.
As
$c-\frac12\sum_{k=1}^d\pt_{x_k}\!b_k\ge 0$ implies
$\langle\LL e^m,e^m\rangle\ge0$, so for $p=2$ one gets
\beq\label{semi_error_intermd}
\kappa_{m,m} \| e^m\|_{L_p(\Omega)}
\le \sum_{j=1}^m (\kappa_{m,j}-\kappa_{m,j-1})\| e^{j-1}\|_{L_p(\Omega)}+\|r^m\|_{L_p(\Omega)}.
\eeq
By \eqref{kappa_def_a}, this implies \eqref{semi_error} for $p=2$.

For $p=\infty$, let $\max_{x\in\Omega}|e^m(x)|=|e^m(x^*)|$ for some $x^*\in\Omega$. Suppose that $e^m(x^*)\ge0$
(the case $e^m(x^*)<0$ is similar). Then $c\ge 0$ implies
$\LL e^m(x^*)\ge 0$,
so \eqref{semidsicr_e_eq} at $x=x^*$
yields
$\kappa_{m,m}e^m(x^*)\le \sum_{j=1}^m(\kappa_{m,j}-\kappa_{m,j-1})e^{j-1}(x^*)+r^m(x^*)$
and then \eqref{semi_error_intermd} for $p=\infty$.
By \eqref{kappa_def_a}, the desired assertion \eqref{semi_error} follows for $p=\infty$.

{\color{blue}
%
Note that in our proof of \eqref{semi_error} for $p=\infty$, we relied on $\LL e^m$ being well-defined in the classical sense.
More generally,  \eqref{new_e_prob_}
implies that $e^m$ solves an elliptic equation with the operator $\LL+\kappa_{m,m}$. Now,
$\{r^j\}\in {L_\infty(\Omega)}$ implies that $e^m\in  C(\bar\Omega)$. So one can modify the above argument
 by using a more general result $\kappa_{m,m}\|e^m\|_{L_\infty(\Omega)}\le\|(\LL+\kappa_{m,m})e^m\|_{L_\infty(\Omega)}$
 (the latter follows from the maximum principle
for functions in $C(\bar\Omega)$  \cite[Corollary~3.2]{GTru}.)
}
%
%
\hfill$\square$

\begin{remark}[More general $\LL$]\label{LL_gen}\color{blue}
The results of this section also apply to a general uniformly-elliptic $\LL$ defined by
$\LL u := \sum_{k=1}^d \bigl\{-\sum_{n=1}^d\pt_{x_k}\!(a_{kn}\,\pt_{x_n}\!u) + b_k\, \pt_{x_k}\!u \bigr\}+c\, u$,
where the coefficients $a_{kn}(x)$ form a symmetric uniformly-positive-definite matrix.
Indeed, when
establishing~\eqref{semi_error}, 
we still have $\langle\LL e^m,e^m\rangle\ge0$  and $\kappa_{m,m}\|e^m\|_{L_\infty(\Omega)}\le\|(\LL+\kappa_{m,m})e^m\|_{L_\infty(\Omega)}$ for, respectively,  $p=2$ and $p=\infty$.
For fully discrete finite-element discretizations, Theorem~\ref{theo_FEM} remains valid, but condition $\mathrm A_p$ for $p=\infty$ may be problematic. Similarly, finite-difference discretizations
that satisfy the discrete maximum principle are not readily available in this more general case.
\end{remark}

\section{Maximum norm error analysis for
finite difference discretizations}\label{sec_FD}

Consider our problem \eqref{problem}--\eqref{LL_def} in the spatial domain $\Omega=(0,1)^d\subset\R^d$.
Let
 $\bar\Omega_h$ be the tensor product of $d$ uniform meshes $\{ih\}_{i=0}^N$,
 with $\Omega_h:=\bar\Omega_h\backslash\pt\Omega$ denoting the set of interior mesh nodes.
Now, consider the finite difference discretization
\beq\label{FD_problem}
\begin{array}{l}
\delta_t^\alpha U^j(z) +\LL_h U^j(z)= f(z,t_j)\quad\mbox{for}\;\;z\in\Omega_h,\;\;j=1,\ldots,M,\\[0.2cm]
U^j=0\quad\mbox{in}\;\;\bar\Omega_h\cap\pt\Omega,\;\;j=1,\ldots,M,\qquad
U^0=u_0\quad\mbox{in}\;\;\bar\Omega_h.
\end{array}
\eeq
Here $\delta_t^\alpha$ is defined by \eqref{delta_def}. The discrete spatial operator $\LL_h$
is a standard finite difference operator defined,
using the standard orthonormal basis $\{\mathbf{i}_k\}_{k=1}^d$ in $\R^d$
(such that $z=(z_1,\ldots,z_d)=\sum_{k=1}^d z_k\, \mathbf{i}_k$ for any $z\in\R^d$), by
\begin{align*}
&\LL_hV(z):=\\[-0.1cm]
&\sum_{k=1}^d h^{-2}\Bigl\{a_k(z+{\textstyle \frac12}h\mathbf{i}_k)\,\bigl[U(z)-U(z+h\mathbf{i}_k)\bigr]+a_k(z-{\textstyle \frac12}h\mathbf{i}_k)\,\bigl[U(z)-U(z-h\mathbf{i}_k)\bigr]\Bigr\}\\[-0.4cm]
&\qquad\quad{}+\sum_{k=1}^d{\textstyle \frac12}h^{-1}\, b_k(z)\,\bigl[U(z+h\mathbf{i}_k)-U(z-h\mathbf{i}_k)\bigr] +c(z)\,U(z)
\quad\qquad\mbox{for}\;\;z\in\Omega_h.
\end{align*}
(Here the terms in the first and second sums respectively discretize $-\pt_{x_k}\!(a_k\,\pt_{x_k}\!u)$ and $b_k\, \pt_{x_k}\!u$ from \eqref{LL_def}.)
The error of this method will be bounded in the nodal maximum norm, denoted
$\|\cdot\|_{\infty\,;\Omega_h}:=\max_{\Omega_h}|\cdot|$.

\begin{theorem}\label{theo_FD}
(i) Let
$\{t_j=T(j/M)^r\}_{j=0}^M$ for some $r\ge1$, and $u$ 
satisfy \eqref{problem}--\eqref{LL_def} 
in $\Omega=(0,1)^d$ with $c\ge 0$.
Then, under the condition
\beq\label{d_max_pr}
h^{-1}\ge \max_{k=1,\ldots,d}\bigl\{{\textstyle\frac12}\|b_k\|_{L_\infty(\Omega)}\,\|a_k^{-1}\|_{L_\infty(\Omega)}\bigr\},
\eeq
there exists a unique solution $\{U^j\}_{j=0}^M$ of \eqref{FD_problem}, and
\beq\label{FD_error_bound}
\|u(\cdot,t_m)-U^m\|_{\infty\,;\Omega_h}\lesssim \max_{j=1,\ldots, m}\|\psi^j\|_{L_\infty(\Omega)}+t_m^{\alpha}\,\|(\LL_h-\LL)u(\cdot,t_m)\|_{\infty\,;\Omega_h},
\eeq
where
$m=1,\ldots,M$, and
$\psi^j=\psi^j(x)$ is defined by \eqref{psi_def}, in which $u(\cdot)$ is understood as $ u(x,\cdot)$  when evaluating $\pt_su$, $\pt_s^2 u$ and $\delta_tu$.
\\
(ii) If $r=1$, then $\max_{j=1,\ldots, m}\|\psi^j\|_{L_\infty(\Omega)}$
in \eqref{FD_error_bound} can be replaced by
a sharper  $\max_{j=1,\ldots, m}\{{\tau^{-\gamma}\,t_j^{1-\alpha+\gamma}}\,\|\psi^j\|_{L_\infty(\Omega)}\}$,
where $\tau=TM^{-1}$ and $\gamma=\min\{\alpha,1-\alpha\}$.
\end{theorem}

\begin{corollary}\label{cor_FD}
(i) Under the conditions of Theorem~\ref{theo_FD}, suppose $\|\pt_t^l u(\cdot,t)\|_{L_\infty(\Omega)}\lesssim 1+t^{\alpha-l}$ for $l=1,2$  and $t\in(0,T]$, and
also $\|\pt_{x_k}^l u(\cdot,t)\|_{L_\infty(\Omega)}\lesssim 1$ for $l=3,4$, $k=1,\ldots,d$ and $t\in(0,T]$.
Then
\begin{align*}
\|u(\cdot,t_m)-U^m\|_{\infty\,;\Omega_h}&\lesssim M^{-\min\{\alpha r,2-\alpha\}}+t_m^{\alpha}\, h^2&\mbox{for}\;\; m=1,\ldots,M.
\intertext{(ii) If, additionally, $r=1$, then  }
\|u(\cdot,t_m)-U^m\|_{\infty\,;\Omega_h}&\lesssim t_m^{\alpha-1}M^{-1}+t_m^{\alpha} \,h^2&\mbox{for}\;\; m=1,\ldots,M.
\end{align*}
\end{corollary}
\begin{proof}
Imitate the proofs of Corollaries~\ref{cor_simplest} and~\ref{cor_simplest_star} for parts (i) and (ii), respectively,
to show that
$|\psi^j|\lesssim M^{-\min\{\alpha r,2-\alpha\}}$
and
$\tau^{-\gamma}\,t_j^{1-\alpha+\gamma}\,|\psi^j|\lesssim \tau$.
Combine these bounds with the standard truncation error estimate $|(\LL_h-\LL)u|\lesssim h^2$.
\end{proof}

\begin{remark}
In the case $d=1$, error bounds similar to those of Corollary~\ref{cor_FD}
can be found in \cite[Theorem~5.2]{stynes_etal_sinum17} and
\cite[Theorem~1]{gracia_etal_cmame} for parts (i) and (ii), respectively.
Note also that the assumptions made in this corollary on the derivatives of $u$ are realistic;
see \S\ref{ss_61} and Example~A in \S\ref{ss_62}.
\end{remark}

\noindent
{\it Proof of Theorem~\ref{theo_FD}.}
For the error $e^m(z):= u(z,t_m)-U^m(z)$, using \eqref{problem} and \eqref{FD_problem}, one easily gets a version of \eqref{err_simplest_prob}:
$$
\delta_t^\alpha e^m +{\color{blue}\LL_h} e^m=R^m:=
\underbrace{\delta_t^\alpha u(\cdot,t_m)-D_t^\alpha u(\cdot,t_m)}_{{}=:r^m}
+(\LL_h-\LL)u(\cdot,t_m)\quad\mbox{in}\;\Omega_h\;\ \mbox{for}\;m\ge1,
$$
subject to $e^0=0$ in $\bar\Omega_h$, and $e^m=0$ on $\bar\Omega_h\cap\pt\Omega$.
Recall that the bound \eqref{r_m_simplest_main} on $r^m$  obtained in the proof of Lemma~\ref{lem_simplest} implies that
$|r^m|\lesssim t_m^{-\alpha}\max_{j=1,\ldots,m}\|\psi^j\|_{L_\infty(\Omega)}$.
Hence, to complete the proof of part (i), it suffices to show that
\beq\label{FD_error}
\delta_t^\alpha \|e^m\|_{\infty\,;\Omega_h}\le \|R^m\|_{\infty\,;\Omega_h}
\qquad\mbox{for}\;\;m=1,\ldots, M.
\eeq
Then, indeed, \eqref{FD_error_bound} immediately follows by an application of Lemma~\ref{lem_main_stability}.

For $r=1$,
%
when dealing with the component $r^m$ of $R^m$, we combine the bound
\eqref{new_rm_bound} (obtained in the proof of Lemma~\ref{lem_simplest_star})
with \eqref{FD_error}
and then employ Lemma~\ref{lem_main_stability_star}, which yields the assertion of part (ii).


To prove \eqref{FD_error}, let $\max_{z\in\Omega_h}|e^m(x)|=|e^m(z^*)|$ for some $z^*\in\Omega_h$.
Suppose that $e^m(z^*)\ge 0$
(the case $e^m(z^*)<0$ is similar). {\color{blue}As \eqref{d_max_pr} combined with $c\ge 0$ implies that
the spatial discrete operator $\LL_h$ is associated with a diagonally-dominant  $M$-matrix}, so
$\LL_h e^m(z^*)\ge 0$,
so $\delta_t^\alpha e^m +\LL e^m=R^m$ at $z=z^*$ 
yields $\delta_t^\alpha e^m(z^*) \le R^m(z^*)$.
In view of \eqref{kappa_def_a}, our assertion \eqref{FD_error} follows.
\hfill$\square$


\section{Error analysis for finite element discretizations}\label{sec_FE}
In this section, we
discretize \eqref{problem}--\eqref{LL_def}, posed in a general bounded  Lipschitz domain  $\Omega\subset\R^d$,
 by applying
 a standard finite element spatial approximation to the temporal semidiscretization~\eqref{semediscr_method}.
 Let $S_h \subset H_0^1(\Omega)\cap C(\bar\Omega)$ be a Lagrange finite element space of fixed degree $\ell\ge 1$ 
 relative to a
 quasiuniform simplicial triangulation
 $\mathcal T$ of $\Omega$.
 (To simplify the presentation, it will be assumed that the triangulation covers $\Omega$ exactly.)
 Now, for $m=1,\ldots,M$, let $u^m_h \in S_h$ satisfy
 \beq\label{FE_problem}
\begin{array}{l}
\langle \delta_t^\alpha u_h^m,v_h\rangle_h +\AA_h (u_h^m,v_h)= \langle f(\cdot,t_m),v_h\rangle_h\qquad\forall v_h\in S_h
\end{array}
\eeq
with $u_h^0= u_0$ or some $u_h^0\approx u_0$.

With $\langle \cdot, \cdot\rangle$ denoting the exact $L_2(\Omega)$ inner product, 
\eqref{FE_problem} employs a possibly approximate inner product $\langle \cdot, \cdot\rangle_h$.
To be more precise, either $\langle \cdot, \cdot\rangle_h=\langle \cdot, \cdot\rangle$, or
$\langle v,w\rangle_h:=\sum_{T\in \mathcal T} Q_T[vw]$
results from an application of a linear quadrature formula $Q_T$ for $\int_T$
with positive weights.
%
Let $\mathring{\AA}$ be the standard bilinear form associated with the elliptic operator $\mathring{\LL}:=\LL-c$ (i.e. $\mathring{\AA}(v,w)=\langle\LL v-cv,w\rangle$ for smooth $v$ and $w$ in $H_0^1(\Omega)$).
The bilinear form $\AA_h$ in \eqref{FE_problem} is related to  $\mathring{\AA}$  and defined by
$\AA_h(v,w):=\mathring{\AA}(v,w)+ \langle cv,w\rangle_h$.

Our error analysis will invoke the Ritz projection $\RR_h u(t)\in S_h$ of $u(\cdot,t)$
associated with our discretization of the operator $\mathring{\LL}$ and
defined by $\mathring{\AA} (\RR_h u,v_h)=\langle\mathring{\LL} u,v_h\rangle_h$ $\forall v_h\in S_h$ and $t\in[0,T]$.


When estimating the error in the $L_p(\Omega)$ norm for $p\in\{2,\infty\}$,
an additional assumption $\mathrm A_p$ will be made, which we now describe.
The set of interior mesh nodes is denoted by $\mathcal N$,  with the corresponding  piecewise-linear hat functions $\{\phi_z\}_{z\in\mathcal N}$.

\begin{enumerate}

\setlength\labelwidth{0.5cm}

\item[$\mathrm A_2$\;] Let $\langle \cdot,\cdot\rangle_h=\langle \cdot,\cdot\rangle$. 
(Otherwise, see Remark~\ref{rem_L2_quadr}).%
 \smallskip

 \item[$\mathrm A_\infty$\!]
 Let $\ell = 1$ 
 (i.e. linear finite elements are employed), and
 let  the stiffness matrix associated with $\AA_h(\cdot,\cdot)+\kappa_{m,m}\langle \cdot, \cdot\rangle_h$
have non-positive off-diagonal entries, i.e.
$\mathbb A^m_{zz'}:= \AA_h(\phi_{z'},\phi_z)+\kappa_{m,m}\langle \phi_{z'},\phi_z\rangle_h\le 0$ for any two interior nodes $z\neq z'$,
where $m=1,\ldots,M$.
\\
(It suffices to check $\mathbb A^m_{zz'}\le 0$ for $m=1$ only,  as $Q_T$ uses positive weights, while $\kappa_{1,1}=\max_{m=1,\ldots,M} \{\kappa_{m,m}\}=\tau_{1}^{-\alpha}/\Gamma(2-\alpha)$.)
 \end{enumerate}

\noindent
Sufficient conditions for  $\mathrm A_\infty$ will be discussed in \S\S\ref{ssec_lumped}--\ref{ssec_lumped3}.
Note  that an assumption similar to  $\mathrm A_\infty$ has been shown to be both necessary and sufficient for non-negativity  preservation in finite element discretizations of equations of type~\eqref{problem}
\cite{laz_thomee}.

\begin{theorem}\label{theo_FEM}
(i) Given $p\in\{2,\infty\}$, let
$\{t_j=T(j/M)^r\}_{j=0}^M$ for some $r\ge1$, and $u$ 
satisfy \eqref{problem}--\eqref{LL_def} 
 with $c-p^{-1}\sum_{k=1}^d\pt_{x_k}\!b_k\ge 0$.
 Then, under the condition $\mathrm A_p$, there exists a unique solution  $\{u_h^m\}_{m=0}^M$ of \eqref{FE_problem}
and, for $m=1,\ldots,M$,
\begin{align}\label{FE_error_bound}
\|u(\cdot,t_m)-u_h^m\|_{L_p(\Omega)}&\lesssim
\|u_0-u_h^0\|_{L_p(\Omega)} + \max_{j=1,\ldots, m}\|\psi^j\|_{L_p(\Omega)}
\\ \notag
&{}+\max_{t\in\{0,t_m\}}\|\rho(\cdot, t)\|_{L_p(\Omega)}+\int_0^{t_m}\|\pt_t \rho(\cdot, t)\|_{L_p(\Omega)}\,dt,
\end{align}
where $\rho(\cdot, t):=\RR_h u(t)-u(\cdot, t)$,
while
$\psi^j=\psi^j(x)$ is defined by \eqref{psi_def}, in which $u(\cdot)$ is understood as $ u(x,\cdot)$  when evaluating $\pt_su$, $\pt_s^2 u$ and $\delta_tu$.
\\
(ii) If $r=1$, then  $\max_{j=1,\ldots, m}\|\psi^j\|_{L_\infty(\Omega)}$
in \eqref{FE_error_bound} can be replaced by
a sharper  $\max_{j=1,\ldots, m}\{{\tau^{-\gamma}\,t_j^{1-\alpha+\gamma}}\,\|\psi^j\|_{L_\infty(\Omega)}\}$,
where $\tau=TM^{-1}$ and $\gamma=\min\{\alpha,1-\alpha\}$.
\end{theorem}

\begin{proof}
Let $e_h^m:=\RR_h u(t_m)-u_h^m\in S_h$.
Then $u(\cdot,t_m)-u_h^m=e_h^m-\rho(\cdot, t_m)$, so
it suffices to prove the desired bounds for $e_h^m$.
Now, a standard calculation using \eqref{FE_problem} and \eqref{problem} yields
\begin{align}\label{prob_for_e}
\langle \delta_t^\alpha e_h^m,&v_h\rangle_h +\AA_h (e_h^m,v_h)
\\[0.2cm]\notag
&=
\langle \delta_t^\alpha \underbrace{\RR_h u}_{=\rho+u}(t_m),v_h\rangle_h +\underbrace{\mathring{\AA} (\RR_h u(t_m),v_h)}_{{}=\langle\mathring{\LL} u(\cdot, t_m),v_h\rangle_h}
+\langle c\RR_h u(t_m)-f(\cdot,t_m),v_h\rangle_h
\\[0.2cm]\notag
&=\langle
\delta_t^\alpha \rho(\cdot, t_m)+c\rho(\cdot, t_m)+ \underbrace{\delta_t^\alpha u(\cdot,t_m)-D_t^\alpha u(\cdot,t_m)}_{{}=:r^m}
,v_h\rangle_h\quad \forall v_h\in S_h,
\end{align}
for $m\ge1$, with $e_h^0=[u_0-u_h^0]+\rho(\cdot,0)$.

Recall that the bound \eqref{r_m_simplest_main} on $r^m$  obtained in the proof of Lemma~\ref{lem_simplest} implies that
$\|r^m\|_{L_p(\Omega)}\lesssim t_m^{-\alpha}\max_{j=1,\ldots,m}\|\psi^j\|_{L_p(\Omega)}$.
Hence, to complete the proof of part~(i), it suffices to show that
\beq\label{FEM_error}
\delta_t^\alpha \|e_h^m\|_{L_p(\Omega)}\le \|\underbrace{ \delta_t^\alpha \rho(\cdot, t_m)+c\rho(\cdot, t_m)+ r^m}_{{}=:R^m}\|_{L_p(\Omega)}
\qquad\mbox{for}\;\;m=1,\ldots, M.
\eeq
%
{\color{blue}Note that $\delta_t^\alpha $ is associated with an $M$-matrix, so
we can 
deal
with the terms  $|\delta_t^\alpha\rho|$ and $|c\rho+r^m|$
in the right-hand side of \eqref{FEM_error} separately.
With this observation,} indeed, \eqref{FE_error_bound} immediately follows by an application of Lemma~\ref{lem_main_stability} when dealing with the term $c\rho+r^m$ in the right-hand side of \eqref{FEM_error},
and Lemma~\ref{lem second_stability} when dealing with $\delta_t^\alpha\rho$.
For the latter, Lemma~\ref{lem second_stability} is applied
with $\lambda^j:=\|\delta_t\rho(\cdot,t_j)\|_{L_p(\Omega)}$. Then $\|\delta_t^\alpha\rho(\cdot,t_m)\|_{L_p(\Omega)}\le J^{1-\alpha}\bar\lambda(t_m)$,
while $\tau_j\,\lambda^j\lesssim \int_{t_{j-1}}^{t_j}\|\pt_t\rho(\cdot,t)\|_{L_p(\Omega)}$,
so the resulting contribution to the bound on $\|e_h^m\|_{L_p(\Omega)}$ will be
$\sum_{j=1}^m\tau_j\,\lambda^j\lesssim \int_{0}^{t_m}\|\pt_t\rho(\cdot,t)\|_{L_p(\Omega)}$.


If $r=1$,
when dealing with the component $r^m$ or $R^m$
in \eqref{FEM_error}, we
recall the bound
\eqref{new_rm_bound} (obtained in the proof of Lemma~\ref{lem_simplest_star})
 and then apply Lemma~\ref{lem_main_stability_star}, which yields the assertion of part (ii).

To prove \eqref{FEM_error}, consider the cases $p=2$ and $p=\infty$ separately.

For $p=2$, set $v_h:=e_h^m$ in \eqref{prob_for_e} and note that
condition $\mathrm A_2$ combined with $c-\frac12\sum_{k=1}^d\pt_{x_k}\!b_k\ge 0$ implies
$\AA_h (e_h^m,e_h^m)
\ge 0$, and then $\langle \delta_t^\alpha e_h^m,e_h^m\rangle\le \langle R^m ,e_h^m\rangle$.
The bound
\eqref{FEM_error} follows in view of \eqref{kappa_def_a}.

For $p=\infty$,
let $\max_{x\in\Omega}|e_h^m(x)|=:|e_h^m(z^*)|$ for some node $z^*\in\mathcal N$.
Now, set $v_h:=\phi_{z^*}$ in \eqref{prob_for_e} and note that
condition $\mathrm A_\infty$ implies
$$
|\AA_h(e^m_h,\phi_{z^*})+\kappa_{m,m}\langle e_h^m,\phi_{z^*}\rangle_h|\ge\bigl\{\AA_h(1,\phi_{z^*})+\kappa_{m,m}\langle 1,\phi_{z^*}\rangle_h \bigr\}\,|e^m_h(z^*)|.
$$
(Here we used the representation $e^m_h=e^m_h(z^*)-\sum_{z\neq z^*}[e^m_h(z^*)-e^m_h(z)]\phi_z$.)
Note also that (in view of the definition of $\AA_h$ related to 
$\LL$ of \eqref{LL_def}) for any $z\in\mathcal N$
$$
\AA_h(1,\phi_{z})+\kappa_{m,m}\langle 1,\phi_{z}\rangle_h
=\langle c+\kappa_{m,m},\phi_{z}\rangle_h
\ge\kappa_{m,m}\langle 1,\phi_{z}\rangle_h\,.
$$
Combining these two observations with \eqref{prob_for_e} and \eqref{kappa_def_a}, we arrive at
$$
\kappa_{m,m}\langle 1,\phi_{z^*}\rangle_h \,\,|e^m_h(z^*)|\le \sum_{j=1}^m \underbrace{(\kappa_{m,j}-\kappa_{m,j-1})}_{{}>0}\,\langle e^{j-1}_h,\phi_{z^*}\rangle_h
+\langle R^m,\phi_{z^*}\rangle_h\,.
$$
Now, recall that $Q_T$ has positive weights 
so
$|\langle v,\phi_{z^*}\rangle_h|\le \|v\|_{L_\infty(\Omega)}\,\langle 1,\phi_{z^*}\rangle_h$ for any~$v$.
With this observation, dividing the above relation by $\langle 1,\phi_{z^*}\rangle_h$ and again using \eqref{kappa_def_a} we finally get \eqref{FEM_error} for $p=\infty$.
\end{proof}

\begin{remark}[Case $\langle \cdot,\cdot\rangle_h\neq\langle \cdot,\cdot\rangle$: error in the $L_2(\Omega)$]\label{rem_L2_quadr}
Suppose
that $Q_T[1]=|T|$ and the Lagrange element nodes in each $T$ are included in the set of quadrature points for $Q_T$, while $h:=\max_{T\in\mathcal T} \{{\rm diam}\,T\}$
is sufficiently small.
Then a version of Theorem~\ref{theo_FEM} is valid for $p=2$ (with condition $A_2$ dropped)
with $\|\cdot\|_{L_2(\Omega)}$ replaced by $\|\cdot\|_{h\,;2}:=\langle \cdot,\cdot\rangle_h^{1/2}$.
Indeed, the proof of Theorem~\ref{theo_FEM} applies to this case 
with  $\AA_h (e_h^m,e_h^m)\ge 0$ for sufficiently small $h$,
in view of
$|\langle c\, e_h^m,e_h^m\rangle_h-\langle c\, e_h^m,e_h^m\rangle)|\lesssim h \|\nabla e_h^m\|_{L_2(\Omega)}$.
Note also that $\|\cdot\|_{h\,;2}\simeq \|\cdot\|_{L_2(\Omega)}$ in $S_h$
(as $\langle \cdot,\cdot\rangle_h$ is an inner product in  $S_h$;
for the latter, note that
$Q_T[v_hw_h]$ generates an inner product for $v_h,w_h\in S_h$ restricted to $T$).
%
\end{remark}
\smallskip

\subsection{Application of Theorem~\ref{theo_FEM} to the error analysis in the  $L_2(\Omega)$ norm}~%
\vspace{-0.3cm}
\label{ss_fem_L2}

\noindent
Let $\Omega\subset\R^d$ (for $d\in\{2,3\}$) be a  domain of polyhedral type as defined in
\cite[\S4.1.1]{MazRssmn}. To be more precise, for $d=3$, the boundary $\pt\Omega$ consists of a finite number of open smooth faces, open smooth edges
and  vertices, the latter being cones with edges. Also, let the angle between any two faces not exceed $\theta^*<\pi$.
(These conditions are satisfied, for example, by a convex domain of polyhedral type, as well as by a smooth domain).
{\color{blue}Then
$\|v\|_{W^2_2(\Omega)}\lesssim \|\LL v\|_{L_2(\Omega)}$; 
 see \cite[Theorem~4.3.2]{MazRssmn} in the case $a_k=1$ $\forall\,k$ in \eqref{LL_def}, as well as \cite[Theorem~5.1]{Kondr67} and \cite[Chapter~4]{Grisvard} for $d=2$.
The treatment of variable smooth coefficients $\{a_k\}$ was addressed in \cite[\S2]{Kondr67}.}

Consequently, for the error of the Ritz projection $\rho(\cdot, t)=\RR_h u(t)-u(\cdot, t)$ one has
\beq\label{Ritz_er_L2}
\|\pt_t^l \rho(\cdot, t)\|_{L_2(\Omega)}\lesssim h\inf_{v_h\in S_h}\|\pt_t^l u(\cdot, t)-v_h\|_{W^1_2(\Omega)}
\quad\;\;\mbox{for}\;\;
l=0,1,\;t\in(0,T].
\eeq
For $l=0$, see, e.g., \cite[Theorem~5.7.6]{BrenScott}.
A similar result for $l=1$ follows as $\pt_t\rho(\cdot, t)=\RR_h \dot u(t)-\dot u(\cdot, t)$, where $\dot u:=\pt_t u$.

\begin{corollary}\label{cor_L2}
(i) Under the conditions of Theorem~\ref{theo_FEM} for $p=2$,
suppose  that $\|\pt_t^l u(\cdot,t)\|_{W^{\ell+1}_2(\Omega)}\lesssim 1+t^{\alpha-l}$ for $l=0,1$
and $\|\pt_t^2 u(\cdot,t)\|_{L_2(\Omega)}\lesssim 1+t^{\alpha-2}$,
where $t\in(0,T]$.
Then
\begin{align*}
\|u(\cdot,t_m)-u_h^m\|_{L_2(\Omega)}&\lesssim M^{-\min\{\alpha r,2-\alpha\}}+
h^{\ell+1}&\mbox{for}\;\; m=1,\ldots,M.
\intertext{(ii) If, additionally, $r=1$, then  }
\|u(\cdot,t_m)-u_h^m\|_{L_2(\Omega)}&\lesssim t_m^{\alpha-1}M^{-1}+
h^{\ell+1}&\mbox{for}\;\; m=1,\ldots,M.
\end{align*}
\end{corollary}
\begin{proof}
Imitate the proofs of Corollaries~\ref{cor_simplest} and~\ref{cor_simplest_star} for parts (i) and (ii), respectively,
to show that
$\|\psi^j\|_{L_2(\Omega)}\lesssim M^{-\min\{\alpha r,2-\alpha\}}$
and
$\tau^{-\gamma}\,t_j^{1-\alpha+\gamma}\,\|\psi^j\|_{L_2(\Omega)}\lesssim \tau\lesssim M^{-1}$.
Combine these bounds with
$\|\pt_t^l \rho(\cdot, t)\|_{L_2(\Omega)}\lesssim h^{\ell+1}(1+t^{\alpha-l})$ for $l=0,1$
(the latter follows from \eqref{Ritz_er_L2}).
\end{proof}

\begin{remark}{\color{blue}%
The assumptions made in Corollary~\ref{cor_L2} on the derivatives of $u$ are realistic;
see \S\ref{ss_61} and Example~B in \S\ref{ss_62}.}
\end{remark}

\begin{remark}
The errors of finite element discretizations of type \eqref{FE_problem} are also estimated in the $L_2(\Omega)$ norm in
a recent paper \cite{laz_L1}, where the authors  particularly address the non-smooth data.
In the case of a uniform temporal mesh and $f=0$, an error bound similar to that of Corollary~\ref{cor_L2}(ii)
is given in \cite[Theorem~3.16(a)]{laz_L1}.
%
\end{remark}

\begin{remark}[Convergence in positive time in the $W_2^1(\Omega)$ semi-norm for $r=1$]
Under condition ${\rm A}_2$,
one has $\AA_h (e_h^m,e_h^m)\ge \|\nabla e_h^m\|^2_{L_2(\Omega)}$.
Now, imitating the proof of \eqref{FEM_error}
for $p=2$, one gets
{
$\delta_t^\alpha\bigl( \kappa_{m,m}^{-1}\|\nabla e_h^m\|^2_{L_2(\Omega)}/\| e_h^m\|_{L_2(\Omega)}+\| e_h^m\|_{L_2(\Omega)}\bigr) \le\| R^m\|_{L_2(\Omega)}$ for $m\ge1$.
Consequently, $\kappa_{m,m}^{-1}\|\nabla e_h^m\|^2_{L_2(\Omega)}/\| e_h^m\|_{L_2(\Omega)}$
(as well as $\| e_h^m\|_{L_2(\Omega)}$)
is bounded similarly to  the error in Corollary~\ref{cor_L2}(ii),
while, by~\eqref{kappa_def_b}, $\kappa_{m,m}\sim M^{\alpha}$.}
%
%
%
%
%
Combining this with 
the standard error bound on $\|\nabla\rho\|_{L_2(\Omega)}$ (see, e.g., \cite[(8.5.4)]{BrenScott}) yields convergence
of \eqref{FE_problem}
in the $W_2^1(\Omega)$ semi-norm for $t_m\gtrsim 1$.
\end{remark}

\subsection{Lumped-mass linear finite elements: application of Theorem~\ref{theo_FEM} to the error analysis in the $L_\infty(\Omega)$ norm}\label{ssec_lumped}~~
\smallskip

\noindent
In this section we restrict our consideration to
the case $a_k=1$ and $b_k=0$ in \eqref{LL_def} for $k=1,\ldots,d$,  and
lumped-mass linear finite-element discretizations,
i.e.
$\ell=1$ and
$\langle \cdot,\cdot\rangle_h$ is defined using the quadrature rule $Q_T[v]:=\int_T v^I$, where $v^I$ is the standard linear Lagrange interpolant.

For the error of the Ritz projection $\rho(\cdot, t)=\RR_h u(t)-u(\cdot, t)$,
one has\vspace{-0pt}
\beq\label{Ritz_er_Linf}
\|\pt_t^l \rho(\cdot, t)\|_{L_\infty(\Omega)}\lesssim h^{2-q}|\ln h|\Bigl\{\|\pt_t^l u(\cdot, t)\|_{W^{2-q}_\infty(\Omega)}
+
\|\pt_t^l \LL u(\cdot, t)\|_{W^{2-q}_{d/2}
(\Omega)}\Bigr\},
\eeq
where $l=0,1$, $q=0,1$ and $t\in(0,T]$.
Consider \eqref{Ritz_er_Linf} for $l=0$ (while the case $l=1$ is similar as $\pt_t\rho(\cdot, t)=\RR_h \dot u(t)-\dot u(\cdot, t)$, where $\dot u=\pt_t u$).
If $\langle\cdot,\cdot\rangle_h=\langle\cdot,\cdot\rangle$, the terms involving $\LL u$
disappear; this version of \eqref{Ritz_er_Linf}
immediately follows from the quasi-optimality of the Ritz projection in the $L_\infty$ norm;
see, e.g.,
\cite[Theorem~2]{Sch_80},
\cite[Theorem~3.1]{Leyk_Vexler_sinum16} and \cite[Theorem~5.1]{SchWa_best},
for, respectively,
polygonal, convex polyhedral and smooth domains.
The lumped-mass quadrature $\langle\cdot,\cdot\rangle_h\neq \langle\cdot,\cdot\rangle$ induces
an additional component $\hat\rho_h\in S_h$ in $\rho$,
defined by
$\langle\nabla\hat\rho_h,\nabla v_h\rangle=\langle \mathring{\LL} u,v_h\rangle_h-\langle
\mathring{\LL} u,v_h\rangle\;\forall v_h\in S_h$.
For completeness, the bound of type \eqref{Ritz_er_Linf} (with $l=0$) for $\hat\rho_h$ is proved in Appendix~\ref{app1}.

As we intend to apply Theorem~\ref{theo_FEM} under condition $\rm A_\infty$,
note that the latter is satisfied
under the following assumptions on the triangulation.
For $\Omega\subset\R^2$, let $\mathcal T$ be a Delaunay triangulation, i.e.,
the sum of the angles opposite to any interior
edge is less than or equal to $\pi$.
In the case $\Omega\subset\R^3$, for any interior edge $E$, let $\omega_E:=\{T\in{\mathcal T}: \pt T\supset E\}$,
and impose that
$\sum_{T\subset\omega_E}|E_T'|\cot\theta_T^E\ge 0$,
where $\theta_T^E$ is the angle between the faces of $T$ not containing $E$, and the edge $E_T'$ is their intersection.
Under these conditions on $\mathcal T$, the stiffness matrix for $-\sum_{k=1}^d\pt^2_{x_k}$ is an $M$-matrix
(see, e.g., \cite[Lemma~2.1]{xu_zik}), while the mass matrix is positive diagonal. So indeed, $\rm A_\infty$ is satisfied.
Note also that it is sufficient, but clearly not necessary, for the triangulation to be non-obtuse (i.e. with no interior angle in any mesh element exceeding $\frac{\pi}2$).

\begin{corollary}\label{cor_L_infty}
(i) Under the conditions of Theorem~\ref{theo_FEM} for $p=\infty$,
suppose  that
$\|\pt_t^l  u(\cdot,t)\|_{W^{2}_\infty(\Omega)}\lesssim 1+t^{\alpha-l}$
and
$\|\pt_t^l  \LL u(\cdot,t)\|_{W^{2}_{d/2}(\Omega)}\lesssim 1+t^{\alpha-l}$
for $l=0,1$,
and also
$\|\pt_t^2 u(\cdot,t)\|_{L_\infty(\Omega)}\lesssim 1+t^{\alpha-2}$,
%
%
where $t\in(0,T]$.
Then
\begin{align*}
\|u(\cdot,t_m)-u_h^m\|_{L_\infty(\Omega)}&\lesssim M^{-\min\{\alpha r,2-\alpha\}}+
h^{2}|\ln h|&\mbox{for}\;\; m=1,\ldots,M.
\intertext{(ii) If, additionally, $r=1$, then  }
\|u(\cdot,t_m)-u_h^m\|_{L_\infty(\Omega)}&\lesssim t_m^{\alpha-1}M^{-1}+
h^{2}|\ln h|&\mbox{for}\;\; m=1,\ldots,M.
\end{align*}
\end{corollary}
\begin{proof}
Imitate the proofs of Corollaries~\ref{cor_simplest} and~\ref{cor_simplest_star} for parts (i) and (ii), respectively,
to show that
$\|\psi^j\|_{L_\infty(\Omega)}\lesssim M^{-\min\{\alpha r,2-\alpha\}}$
and
$\tau^{-\gamma}\,t_j^{1-\alpha+\gamma}\,\|\psi^j\|_{L_\infty(\Omega)}\lesssim \tau\lesssim M^{-1}$.
Combine these bounds with
$\|\pt_t^l \rho(\cdot, t)\|_{L_\infty(\Omega)}\lesssim h^{2}|\ln h|(1+t^{\alpha-l})$ for $l=0,1$
(the latter follows from \eqref{Ritz_er_Linf}).
\end{proof}

\begin{remark}
%
The assumptions made in Corollary~\ref{cor_L_infty} on the derivatives of $u$ are realistic;
see \S\ref{ss_61} and Example~C in \S\ref{ss_62}.
\end{remark}
\smallskip

\subsection{Linear finite elements without quadrature: a comment on 
the error analysis in the $L_\infty(\Omega)$ norm}~~\label{ssec_lumped3}
\smallskip

\noindent
{\color{blue}We shall start by checking condition $\rm A_\infty$, used in Theorem~\ref{theo_FEM}, for the simplest case
of $d=1$ and  $\LL=-\pt^2_{x_1}$.
A straightforward calculation shows
that the stiffness matrix associated with $\AA_h(\cdot,\cdot)+\kappa_{1,1}\langle \cdot, \cdot\rangle$ will be
a tridiagonal matrix with diagonal entries $\frac2h+\frac23 h\kappa_{1,1}$
and off-diagonal entries $-\frac1h+\frac16 h\kappa_{1,1}$.
So, for off-diagonal entries to be non-positive, one needs to impose $\kappa_{1,1}\le 6 h^{-2}$, i.e.
$\tau_1^{-\alpha}\le 6\Gamma(2-\alpha)h^{-2}$. Note that exactly the same condition  is required for the discrete maximum principle
in the classical parabolic case of \eqref{problem} with $\alpha=1$
assuming the backward Euler discretization in time is combined with linear finite elements without quadrature.

A similar condition is true if
}
 $a_k=1$  in \eqref{LL_def} for $k=1,\ldots,d$, and
$\langle \cdot,\cdot\rangle_h=\langle \cdot,\cdot\rangle$. Then
the mass matrix is not diagonal and contains positive off-diagonal entries.
Still, condition $\rm A_\infty$ is satisfied
(and so Theorem~\ref{theo_FEM} with $p=\infty$ can be applied)
 if $h^2\tau_1^{-\alpha}\le C_{\mathcal T}$
for a sufficiently small constant $C_{\mathcal T}$ that we specify below,
and, additionally,
the triangulation is non-obtuse and $\min_{T\subset\omega_E}\theta_T^E\le \theta^*$ for some fixed positive $\theta^*<\frac{\pi}2$
(for the notation, see \S\ref{ssec_lumped}).
Indeed, for such a triangulation, not only the stiffness matrix for $-\sum_{k=1}^d\pt^2_{x_k}$ is an $M$-matrix,
but its contribution to $\mathbb A^m_{zz'}$, for any two nodes $z\neq z'$ connected by an interior edge $E$,
will be strictly negative and equal to
$ -\sum_{T\subset\omega_E}|E_T'|\cot\theta_T^E/\{d(d-1)\}$
(with $E_T'$, in the case $d=2$, being a node and the notational convention $|E_T'|=1$ used);
see \cite[Lemma~2.1]{xu_zik}.
A calculation also shows that
the contribution of
$\langle (\kappa_{1,1}+c)\phi_{z'},\phi_z\rangle$ to $\mathbb A^m_{zz'}$ does not exceed
$(\tau_1^{-\alpha}/\Gamma(2-\alpha)+\|c\|_{L_\infty(\Omega)})|\omega_E|/\{(d+1)(d+2)\}$.
Furthermore,
the contribution of
$\langle  b_k(x)\, \pt_{x_k}\!\phi_{z'},\phi_z\rangle$ to $\mathbb A^m_{zz'}$ is $\lesssim h^{-1}|\omega_E|$.
As the triangulation is  quasi-uniform, these observations imply that 
there is a positive constant $C'_{\mathcal T}$ such that for any interior edge $E$, one has
$$
\frac{(d+1)(d+2) }{d(d-1)}\,\,\,|\omega_E|^{-1}\!\!\sum_{T\subset\omega_E}|E_T'|\cot\theta_T^E \ge C'_{\mathcal T} h^{-2}.
\vspace{-4pt}
$$
Now, $h^2\tau_1^{-\alpha}\le C_{\mathcal T}$,
with any fixed  constant $C_{\mathcal T}< C'_{\mathcal T}\Gamma(2-\alpha)$,
implies $\rm A_\infty$ (assuming that $h$ is sufficiently small; in fact, one can use $C_{\mathcal T}=C'_{\mathcal T}\Gamma(2-\alpha)$ if
$c=0$ and $b_k=0$ for $k=1,\ldots,d$ in \eqref{LL_def}).
To avoid computing $C'_{\mathcal T}$, one can instead impose $h^2|\ln h|\,\tau_1^{-\alpha}\le C_{\mathcal T}$ with any fixed $C_{\mathcal T}>0$
and $h$ sufficiently small.
Note that although the above triangulation condition is somewhat restrictive,
it is satisfied
by mildly structured meshes
with all mesh elements close
to equilateral triangles/regular tetrahedra.

Note also that in most practical situations, the convergence rates do not deteriorate because of the restriction $\tau_1^\alpha\gtrsim h^2$.
To be more precise, as long as $r\le (2-\alpha)/\alpha$ (including the optimal $r= (2-\alpha)/\alpha$),
the error in part (i) of Corollary~\ref{cor_L_infty} is $\lesssim M^{-\alpha r}+h^2|\ln h|\sim \tau_1^{\alpha}+h^2|\ln h|$.
Similarly, in part (ii) for $t_m\gtrsim 1$, the error is $\lesssim \tau_1+h^2|\ln h|$, so a reasonable choice $\tau_1\sim h^2$ is clearly within the restriction
$\tau_1^\alpha\gtrsim h^2$.

\section{Estimation of derivatives of the exact solution $u$}\label{sec_dervts}

The purpose of this section is to show that the assumptions made
in \S\S\ref{sec_semidiscr}--\ref{sec_FE}
on the derivatives of the exact solution $u$ of \eqref{problem} are realistic, and
give examples of when they are satisfied.
The discussion will be mainly restricted to the case of the operator $\LL$ being self-adjoint (i.e. $b_k=0$ for $k=1\ldots,d$ in \eqref{LL_def});
for the non-self-adjoint case, see Remark~\ref{rem_nonsef_LL} below.
For simplicity, we also assume that $\Omega$ is either
a convex domain of polyhedral type or a smooth domain.
Hence, we shall be able to invoke
$\|v\|_{W^2_2(\Omega)}\lesssim \|\LL v\|_{L_2(\Omega)}$ when $v=0$ on $\pt\Omega$,
as well as the consequent property $\|v\|_{L_\infty(\Omega)}\lesssim \|\LL v\|_{L_2(\Omega)}$
(in view of the Sobolev embedding theorem).

The approach that we consider here employs the method of separation of variables, in which
the eigenvalues and eigenfunctions of the self-adjoint operator $\LL$
(see, e.g., \cite[\S6.5]{evans}  for their existence and properties)
 are used to get
an explicit eigenfunction expansion of $u$.
Note that the time-dependent coefficients in this expansion are represented using Mittag-Leffler functions.
This approach was used in
  \cite{sakamoto} for smooth domains, \cite[\S2.2 and \S3.4]{laz_semidiscr} for polygonal/polyhedral domains,
and \cite[\S2]{stynes_etal_sinum17}
 for $\Omega=(0,1)$.
{\color{blue}Eigenfunction expansions are frequently used to establish regularity estimates for fractional-derivative problems; see, e.g.
\cite{McL10, NOS16}, where somewhat different problems were considered. In particular, the bounds \cite[(1.6) and (1.7)]{McL10}
are somewhat similar to those we obtain below.}

\subsection{Temporal derivatives of $u$}\label{ss_61}
The assumptions made in Corollary~\ref{cor_semidiscr} on temporal derivatives of $u$ (that $\|\pt_t^l u(\cdot,t)\|_{L_p(\Omega)}\lesssim 1+t^{\alpha-l}$ for $l=1,2$, and $p\in\{2,\infty\}$) are realistic.
For example,  for the case  $p=\infty$, $d=1$ and $\LL=-\pt_{x_1}^2+c(x_1)$, they are  satisfied under
certain regularity assumptions on $u_0$ and $f$ (including $\LL^l f(\cdot,t)=\LL^q u_0=0$ on $\pt\Omega$ for $l=0,1$ and $q=0,1,2$) by \cite[Theorem~2.1]{stynes_etal_sinum17}.
The proof relies on the term-by-term differentiation with respect to $t$ of the eigenfunction expansion of $u$.
Note that this proof cannot be directly extended to $d>1$
(as the eigenfunctions are not necessarily  uniformly bounded, while the eigenvalues exhibit a different asymptotic behaviour in higher dimensions).

These difficulties are avoided by the following modification.
A term-by-term application of $\LL^q\pt_t^l$ to the eigenfunction expansion of $u$ yields
$\|\LL^q\pt_t^l u(\cdot,t)\|_{L_2(\Omega)}\lesssim 1+t^{\alpha-l}$ for $l=1,2$ and $q=0,1$.
Now, setting $q=0$ and $q=1$ implies the desired bounds on the temporal derivatives  for $p=2$ and $p=\infty$, respectively.
It should be noted that this approach relies on the regularity assumptions that
$\|u_0\|_{\LL^{q+2}}+\|\pt_t^lf(\cdot,t)\|_{\LL^{q+1}}\lesssim 1$ for 
$l=0,1,2$
(where the assumptions of the temporal derivatives of $f$ may, in fact, be weakened).
Here (similarly to  \cite{laz_semidiscr,sakamoto,stynes_etal_sinum17}) we used the norm
$\|v\|_{\LL^\gamma}:=\bigl\{\sum_{i=1}^\infty\lambda_i^{2\gamma}\langle v,\psi_i\rangle^2\bigr\}^{1/2}$,
where $0<\lambda_1<\lambda_2\le \lambda_3\le\ldots$ are the eigenvalues of $\LL$,
while $\{\psi\}_{i=1}^\infty$ are the corresponding normalized eigenfunctions satisfying $\|\psi_i\|_{L_2(\Omega)}=1$.
\vspace{-0.1cm}

\subsection{Spatial and mixed derivatives of $u$}\label{ss_62}
In \S\S\ref{sec_FD}--\ref{sec_FE} (see Corollaries \ref{cor_FD},\,\ref{cor_L2},\,\ref{cor_L_infty}), a number of additional assumptions were made that involve spatial derivatives of~$u$.
Here the situation is more delicate, as if $\Omega$ has any corners, $u$ may exhibit corner singularities.

{\it Example A.}
Consider $\Omega=(0,1)^2$ and $\LL=-[\pt_{x_1}^2+\pt_{x_2}^2] +c(x_1,x_2)$ 
under the assumption
$\|u_0\|_{\LL^{3}}+\|f(\cdot,t)\|_{\LL^{5/2}}\lesssim 1$.
Note that the latter implies that the elliptic corner compatibility conditions up to order 2 are satisfied.
Hence, \cite[Theorem~3.1]{volkov} combined with the Sobolev embedding theorem yields
$\|u\|_{W^4_\infty(\Omega)}\lesssim \|\LL u\|_{W^{2+\epsilon}_\infty(\Omega)}\lesssim \|\LL u\|_{W^4_2(\Omega)}$ for any $t\in(0,T]$.
Similarly, $\|\LL u\|_{W^4_2(\Omega)}\lesssim \|\LL^2 u\|_{W^2_2(\Omega)}\lesssim \|\LL^3 u\|_{L_2(\Omega)}$,
while one can show (by an application of $\LL^3$ to the eigenfunction expansion of $u$)
that $\|\LL^3 u\|_{L_2(\Omega)}\lesssim 1$. Combining these observations, one gets
$\|u\|_{W^4_\infty(\Omega)}\lesssim 1$,
so the assumptions made in Corollary~\ref{cor_FD} on the spatial derivatives of $u$ are satisfied.%


{\it Example B.}
It is assumed in Corollary~\ref{cor_L2}
  that $\|\pt_t^l u(\cdot,t)\|_{W^{\ell+1}_2(\Omega)}\lesssim 1+t^{\alpha-l}$ for $l=0,1$  and $t\in(0,T]$.
For linear finite elements, i.e. $\ell=1$,
these bounds follow from
$\|\pt_t^l u\|_{W^2_2(\Omega)}\lesssim \|\LL\pt_t^l u\|_{L_2(\Omega)}$
combined
with the bound on $\|\LL\pt_t^l u(\cdot,t)\|_{L_2(\Omega)}$ obtained in \S\ref{ss_61}
(see the case $q=1$).
For $\ell>1$, a similar argument can be used  (under additional data regularity assumptions) if $\Omega$ is smooth.%

{\it Example C.}
If $\Omega$ is smooth,
then both $\|\pt_t^l u(\cdot,t)\|_{W^{2}_\infty(\Omega)}$ and $\|\LL\pt_t^l u(\cdot,t)\|_{W^{2}_{d/2}(\Omega)}$ are
$\lesssim \|\LL\pt_t^l u(\cdot,t)\|_{W^{2}_2(\Omega)}$.
For the latter, using the argument of Example B,
one can show that
$\|\LL\pt_t^l u(\cdot,t)\|_{W^{2}_2(\Omega)}\lesssim 1+t^{\alpha-l}$ for $l=0,1$
under the regularity assumption
$\|u_0\|_{\LL^{3}}+\|\pt_t^lf(\cdot,t)\|_{\LL^{2}}\lesssim 1$.
So for this example, the assumptions made in Corollary~\ref{cor_L_infty} on $u$ are satisfied.
\vspace{-0.2cm}

%

\begin{remark}[Non-self-adjoint $\LL$]\label{rem_nonsef_LL}
Even if some coefficient(s) $b_k\neq 0$ in \eqref{LL_def},
one can sometimes employ the eigenfunction expansion after
reducing the problem \eqref{problem} to the self-adjoint case.
For example, if the coefficients $\{a_k\}$ and $\{b_k\}$ in \eqref{LL_def} are constant,
it suffices to rewrite \eqref{problem} for the unknown function
$\widetilde u:=u\exp\bigl\{-\sum_{k=1}^d\frac12 (b_k/a_k)x_k\bigr\}$.
A similar trick for the case of variable coefficients and $d=1$ is described in \cite[\S2]{gracia_etal_cmame}.
\end{remark}

\newpage

\section{Numerical results}\label{sec_Num}

  \begin{figure}[t!]
\begin{center}
\includegraphics[height=0.38\textwidth]{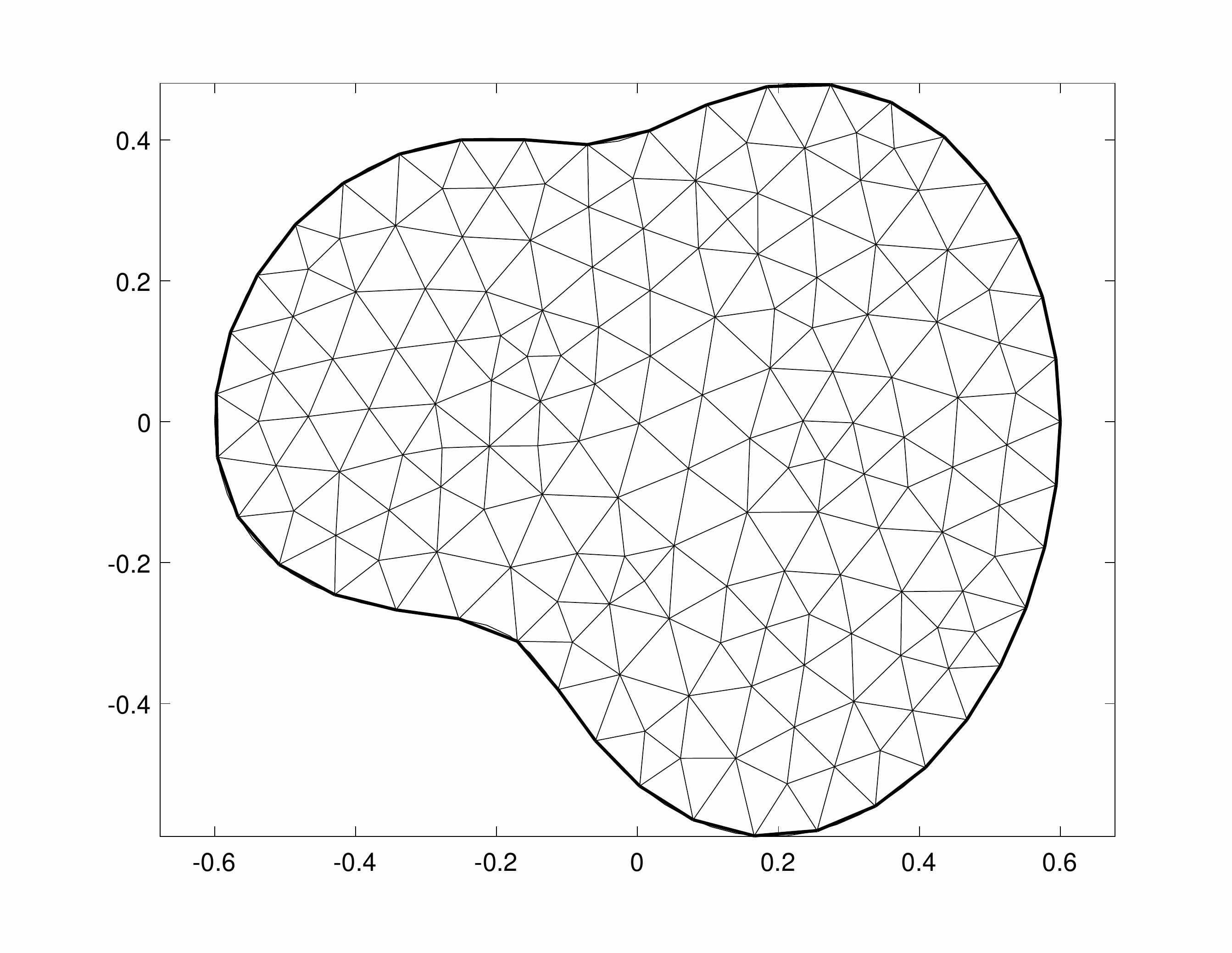}\hfill\includegraphics[height=0.38\textwidth]{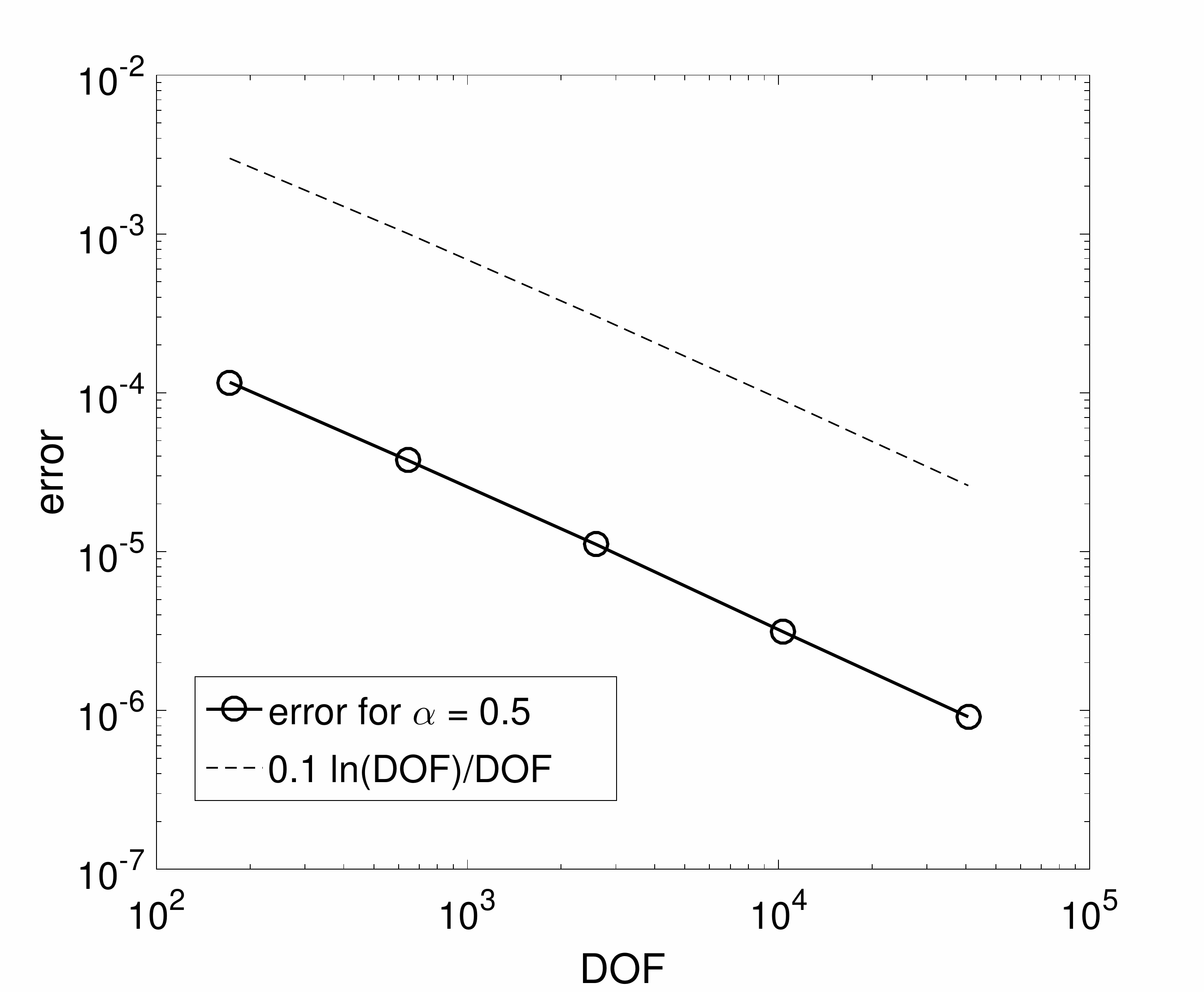}
\end{center}
\vspace{-0.3cm}
 \caption{\label{fig_mesh}\it\small
 Delaunay triangulation of $\Omega$ with DOF=172 (left),
 maximum nodal errors for $\alpha=0.5$, $r=(2-\alpha)/\alpha$ and $M=10^4$.}
 \end{figure}

 {
\begin{table}[b!]
\begin{center}
\caption{
 Maximum nodal errors (odd rows) and
computational rates $q$ in $M^{-q}$ (even rows) for
 $r=(2-\alpha)/\alpha$ and spatial DOF=398410}
\label{t2}
\vspace{-0.1cm}
{\small
\begin{tabular}{rrrrrrrr}
\hline
\strut\rule{0pt}{9pt}&&
$M=64$& $M=128$&$M=256$&
$M=512$& $M=1024$&$M=2048$\\
\hline
$\alpha=0.3$&
&4.157e-4	&1.428e-4	&4.750e-5	&1.558e-5	&5.053e-6& 1.624e-6\\
&&1.542	&1.588	&1.608	&1.624	&1.637\\[3pt]
%
%
%
$\alpha=0.5$&
&7.824e-4	&3.109e-4	&1.173e-4	&4.301e-5	&1.555e-5	&5.582e-6\\	
&&1.331	&1.407	&1.447	&1.468	&1.478	\\[3pt]
%
%
$\alpha=0.7$&
&1.236e-3	&5.924e-4	&2.693e-4	&1.181e-4	&5.045e-5&2.120e-5\\
&&1.061	&1.137	&1.190	&1.226  &1.251\\
\hline
\end{tabular}}
\end{center}
\end{table}
}

Our model problem is \eqref{problem} with $\LL=-(\pt_{x_1}^2+\pt_{x_2}^2)$, posed
in the domain $\Omega\times[0,1]$ (see Fig.\,\ref{fig_mesh}, left)
with $\pt\Omega$  parameterized by
$x_1(l):=\frac23R\cos\theta$ and $x_2(l):=R\sin\theta$, where
$R(l): =  0.4 + 0.5\cos^2\! l$ and
$\theta(l) := l + e^{(l-5)/2}\sin(l/2)\sin l$ for $l\in[0,2\pi]$.
We choose  $f$, as well as the initial and non-homogeneous boundary conditions, so that the unique exact solution
$u=t^{\alpha}\cos(xy)$.
This problem is discretized by \eqref{FE_problem} (with an obvious modification for the case of non-homogeneous boundary conditions)
using
lumped-mass linear finite elements (described in \S\ref{ssec_lumped})
on quasiuniform Delaunay triangulations of $\Omega$ (with DOF denoting the number of degrees of freedom in space).

The errors in the maximum nodal norm
$\max_{z\in{\mathcal N},\,m=1,\ldots, M}|u_h^m(z)-u(z,t_m)|$
are shown in
Fig.\,\ref{fig_mesh} (right) and
Table~\ref{t2}
for, respectively, a large fixed $M$ and DOF.
In the latter case, we also give computational rates of convergence.
The graded temporal mesh
$\{t_j=T(j/M)^r\}_{j=0}^M$ was used
with the optimal $r= (2-\alpha)/\alpha$  (see Remark~\ref{rem_r_optimal}).
By
Corollary~\ref{cor_L_infty}(i),
the errors are expected to be $\lesssim M^{-(2-\alpha)}+h^2|\ln h|$.
Our numerical results clearly confirm the sharpness of this corollary for the considered case.
For more extensive numerical experiments,  we refer the reader to
\cite{stynes_etal_sinum17},
where, in particular, the influence of  $r$ on the errors is numerically investigated,
as well as  \cite{gracia_etal_cmame,laz_L1} for numerical results on uniform temporal meshes.

\appendix

\section{Proof of Lemma~\ref{lem_main_stability_star}}
\label{app_B}

\begin{proof}
(i) First, consider $\gamma=\alpha$.
As the operator $\delta_t^\alpha$ is associated with an $M$-matrix, it suffices to construct a barrier function $0\le B^j\lesssim t_j^{\alpha-1}$
such that $\delta_t^\alpha B^j\gtrsim \tau^\alpha t_j^{-\alpha-1}$.
Fix a sufficiently large number $2\le p\lesssim 1$, and then
set $\beta:=1-\alpha$ and $B(s):=\min\bigl\{(s/t_p)t_p^{-\beta}, s^{-\beta} \bigr\}$, and also $B^j:=B(t_j)$.
Note that, when using the notation of type $\lesssim$, the dependence on $p$ will be shown explicitly.

For $j\le p$, a straightforward calculation shows that
$\delta_t^\alpha B^j= D^\alpha_t B(t_j)\sim t_j^\beta t_p^{-\beta-1}\gtrsim p^{-\beta-1}(\tau^\alpha t_j^{-\alpha-1})$.
%
%
Next, for $D_t^\alpha B(t)$ with $t>t_p$ one has
\begin{align*}
\Gamma(1-\alpha)\,D^\alpha_t B(t)=&\underbrace{\int_0^{t_p}\!\! t_p^{-\beta-1}(t-s)^{-\alpha}\,ds}_%
{{}\ge t_p^{-\beta} t^{-\alpha}}
-\underbrace{\beta\int_{t_p}^t\! s^{-\beta-1}(t-s)^{-\alpha}\,ds}_%
{{}=:  t^{-1}{ I} }\,.
\end{align*}
Here, using $\hat s:=s/t$ and $\hat t_p:=t_p/t$, and noting that $\alpha+\beta=1$, one gets
$$
{I}=\beta\int_{\hat t_p}^1 \hat s^{-\beta-1}(1-\hat s)^{-\alpha}\,d\hat s
=\hat t_p^{-\beta}(1-\hat t_p)^{\beta}
\le \hat t_p^{-\beta}(1-\beta \hat t_p).
$$
Now, using $t^{-1}\hat t_p^{-\beta}=t_p^{-\beta} t^{-\alpha}$, one concludes for $t>t_p$ that
\beq\label{app_B_eq}
\Gamma(1-\alpha)\,
D^\alpha_t B(t)
\ge
t_p^{-\beta} t^{-\alpha}\,(\beta t_p/t)
=
\beta t_p^{\alpha}t^{-\alpha-1}=\beta p^{\alpha}\,(\tau^\alpha t^{-\alpha-1}).
\eeq
So, to complete the proof,
 it remains to show that $\frac12 D^\alpha_t B(t_m)\ge|\delta^\alpha_t B^m-D^\alpha_t B(t_m)|$ for any $m> p$.
 For the latter,
 with the notation $F(s):=\beta^{-1}(t_m-s)^{\beta}$,
note that $\delta^\alpha_t B^m$ involves
$\sum_{j=1}^m$ of the
 the terms
$$
\delta_t B^j \int_{t_{j-1}}^{t_j}\!\!\underbrace{(t_m-s)^{-\alpha}}_{{}=F'(s)}ds
=\delta_t F^j\int_{t_{j-1}}^{t_j}\!\!\!B'(s)ds.
$$
Note also that the component $\sum_{j=1}^p$ is identical in $\delta^\alpha_t B^m$ and $D^\alpha_t B(t_m)$.
Now, subtracting one of the above representations from the corresponding components  $\int_{t_{j-1}}^{t_j}\!B'(s)F'(s)\,ds$ of $D^\alpha_t B(t_m)$ yields
$$
|\delta^\alpha_t B^m-D^\alpha_t B(t_m)|\lesssim
\tau\int_{t_p}^{t_n}\!\!\!
s^{-\beta-1}(t_{m-1}-s)^{-\alpha-1}\,ds
+
\tau\int_{t_n}^{t_m}\!\!\!(s-\tau)^{-\beta-2}(t_m-s)^{-\alpha}ds,
$$
where $n:=\max\{p,\,\lfloor m/2\rfloor\}$, and $n\le m-2$ whenever $n>p$.
Here, when dealing with $s\in(t_{j-1},t_j)$, we also used
$|\delta_t F^j-F'(s)|\le \tau|F''(t_{j})|\lesssim \tau (t_{m-1}-s)^{-\alpha-1}$
for $j\le n$, and
 $| \delta_t B^j-B'(s)|\le \tau|B''(t_{j-1})|\lesssim \tau (s-\tau)^{-\beta-2}$
for $j>n$.
Estimating the above integrals $\int_{t_p}^{t_n}$ and $\int_{t_n}^{t_m}$
similarly to $I$ and respectively using
$(t_{m-1}-s)^{-1}\le 2(t_{m}-s)^{-1}$
and
$(s-\tau)^{-1}\le 2s^{-1}$, one finally gets
$$
|\delta^\alpha_t B^m-D^\alpha_t B(t_m)|\lesssim \tau t_m^{-2}(t_p/t_m)^{-\beta}
\sim p^{-\beta}\,(\tau^\alpha t_m^{-\alpha-1}).
$$
Combining this with \eqref{app_B_eq} and choosing $p$ sufficiently large yields
the desired assertion $\delta^\alpha_t B^m\gtrsim\tau^\alpha t_m^{-\alpha-1}$.
\smallskip

{\color{blue}(ii) It remains to consider $\gamma\in(0,\alpha)$. Set
$p_m:=2^m p$ and $c_m:=2^{-m\gamma}$.
Now, set
 $B_m(s):=\min\bigl\{s t_{p_m}^{-\beta-1}, s^{-\beta} \bigr\}$
 (i.e. $B_0(s)=B(s)$),
 and  $B_m^j:=B_m(t_j)$, and then $\bar B^j:=\sum_{m=0}^\infty c_mB_m^j$.
 Here $p$ is from part (i),
and, when using the notation of type $\lesssim$, the dependence on $\gamma$ and $m$, but not on $p$, will be shown explicitly.

 Imitating the argument used in part (i), one gets $\delta_t^\alpha B_m^j\ge 0$ for $j\ge 0$, while
 for $j>p_m$ one has
 $\delta_t^\alpha B_m^j\gtrsim t_{p_m}^\alpha t_j^{-\alpha-1}$
 (compare with \eqref{app_B_eq}).
The latter implies
 $c_m(\delta_t^\alpha B_m^j)\gtrsim c_m t_{p_m}^\gamma t_j^{-\gamma-1}\ge \tau^\gamma t_j^{-\gamma-1}$ for $p_m< j\le p_{m+1}$.
Combining this with
$c_0=1$ and
$\delta_t^\alpha B_0^j\gtrsim \tau^\alpha t_j^{-\alpha-1}\gtrsim \tau^\gamma t_j^{-\gamma-1}$ for $1\le j\le p_0$, one concludes that
$\delta_t^\alpha \bar B^j\gtrsim \tau^\gamma t_j^{-\gamma-1}$.
Finally, note that
$\sum_{m=0}^\infty c_m= C_\gamma:=(1-2^{-\gamma})^{-1}$,
so
$\bar B^j\le  C_\gamma t_j^{-\beta}=C_\gamma t_j^{\alpha-1}$, which   completes the proof.}
\end{proof}

\section{Lumped-mass quadrature error in the maximum norm}\label{app1}

The lumped-mass quadrature $\langle\cdot,\cdot\rangle_h\neq \langle\cdot,\cdot\rangle$ induces
an additional component $\hat\rho_h\in S_h$ in the error of the Ritz projection $\rho(\cdot, t)=\RR_h u-u$,
defined by
$\langle\nabla\hat\rho_h,\nabla v_h\rangle=\langle \mathring{\LL} u,v_h\rangle_h-\langle
\mathring{\LL} u,v_h\rangle\;\forall v_h\in S_h$.
We claim that 
\beq\label{Ritz_er_Linf__}
\| \hat\rho_h\|_{L_\infty(\Omega)}\lesssim h^{2-q}|\ln h|\,
\| \mathring{\LL} u(\cdot, t)\|_{W^{2-q}_{d/2}(\Omega)}
\qquad\mbox{for}\; q=0,1.
\eeq
The desired bound of type \eqref{Ritz_er_Linf} (with $l=0$) for $\hat\rho_h$ follows in view of $\mathring{\LL}=\LL-c$.

To prove \eqref{Ritz_er_Linf__}, a standard calculation yields, for any $v_h\in S_h$ and $q=0,1$, 
$$
|\langle\nabla\hat\rho_h,\nabla v_h\rangle|
\lesssim h^{2-q} \Bigl\{
\|\mathring{\LL} u\|_{W^{2-q}_{d/2}(\Omega)}\|v_h\|_{L_{d/(d-2)}(\Omega)}+\|\mathring{\LL} u\|_{W^{1-q}_d(\Omega)}\|\nabla v_h\|_{L_{d/(d-1)}(\Omega)}
\Bigr\}.
$$
In view of the Sobolev embedding $\|\mathring{\LL} u\|_{W^{1-q}_d(\Omega)}\lesssim \|\mathring{\LL} u\|_{W^{2-q}_{d/2}(\Omega)}$, one arrives at
\beq\label{appendix_eq}
|\langle\nabla\hat\rho_h,\nabla v_h\rangle|
\lesssim h^{2-q} \Bigl\{
\|v_h\|_{L_{d/(d-2)}(\Omega)}+\|\nabla v_h\|_{L_{d/(d-1)}(\Omega)}
\Bigr\}\,\|\mathring{\LL} u\|_{W^{2-q}_{d/2}(\Omega)}.
\eeq
Next, consider the cases $d=2,3$ separately.

For $d=2$, one has $d/(d-2)=\infty$ and $d/(d-1)=2$.
Set $v_h:=\hat\rho_h$ in \eqref{appendix_eq}, and recall the discrete Sobolev inequality $\|\hat\rho_h\|_{L_{\infty}(\Omega)}\lesssim |\ln h|^{1/2}\|\nabla \hat\rho_h\|_{L_2(\Omega)}$,
so
$\|\nabla \hat\rho_h\|_{L_2(\Omega)}
\lesssim h^{2-q} |\ln h|^{1/2}\,\|\mathring{\LL} u\|_{W^{2-q}_{d/2}(\Omega)}$, so \eqref{Ritz_er_Linf__} follows.

For $d=3$,
with $\|\hat\rho_h\|_{L_\infty(\Omega)}=|\hat\rho_h(x^*)|$ for some interior node $x^*\in\mathcal N$, let $g_h\in S_h$ be
a discrete version of the Green's function
 $g_h\in S_h$  associated with $x^*$ and defined by
$\langle\nabla g_h,\nabla v_h\rangle= v_h(x^*)$ $\forall v_h\in S_h$.
Now set $v_h:=g_h$ in \eqref{appendix_eq}, so
$$
\|\hat\rho_h\|_{L_\infty(\Omega)}=|\langle\nabla \hat\rho_h,\nabla g_h\rangle|
\lesssim h^{2-q} \Bigl\{
\underbrace{\|g_h\|_{L_{3}(\Omega)}}_{{}\lesssim |\ln h|^{1/3}}
+
\underbrace{\|\nabla g_h\|_{L_{3/2}(\Omega)}}_{{}\lesssim |\ln h|^{2/3}}
\Bigr\}\,\|\mathring{\LL} u\|_{W^{2-q}_{d/2}(\Omega)},
$$
where we employed the bounds on $\|g_h\|_{L_{3}(\Omega)}$ and $\|\nabla g_h\|_{L_{3/2}(\Omega)}$
from
\cite[see (3.10), (3.11) and the final formula in \S3]{Leyk_Vexler_sinum16}.
So we again get \eqref{Ritz_er_Linf__}.

\end{document}